\documentclass[11pt]{article}

\usepackage{amsmath}
\usepackage{amsfonts} 
\usepackage{mathtools} 
\usepackage{esint}
\usepackage{xcolor}
\usepackage{cancel}
\usepackage{soul}
\usepackage{bm}
\usepackage{bbm}
\usepackage{amsthm}
\usepackage{enumerate}
\usepackage{cleveref}
\usepackage{color}

\newcommand{\R}{\mathbb R}

\newcommand{\1}{\mathbbm{1}}

\newcommand{\dv}{\mathrm{d}v}
\newcommand{\dvs}{\mathrm{d}v_*}
\newcommand{\vsi}{{\vs}_i}
\newcommand{\vsj}{{\vs}_j}
\newcommand{\vsk}{{\vs}_k}

\newcommand{\vi}{v_i}
\newcommand{\vj}{v_j}
\newcommand{\vk}{v_k}

\newcommand{\nablas}{\nabla_*}

\newcommand{\nablav}{\nabla_v}

\newcommand{\partialvi}{\partial_{\vi}}

\newcommand{\partialvsi}{\partial_{\vsi}}

\newcommand{\partialvivj}{\partial_{\vi\vj}}

\newcommand{\partialvsivsj}{\partial_{\vsi\vsj}}

\newcommand{\dsigma}{\mathrm{d}\sigma}

\newcommand{\dtheta}{\mathrm{d}\theta}
\newcommand{\domega}{\mathrm{d}\omega}
\newcommand{\dphi}{\mathrm{d}\phi}
\newcommand{\dmu}{\mathrm{d}\mu}
\newcommand{\dnu}{\mathrm{d}\nu}
\newcommand{\dx}{\mathrm{d}x}
\newcommand{\du}{\mathrm{d}u}

\newcommand{\ds}{\mathrm{d}s}

\newcommand{\vphi}{\varphi}
\newcommand{\vphis}{\varphi_*}

\newcommand{\dt}{\mathrm{d}t}

\newcommand{\eps}{\varepsilon}

\newcommand{\bde}{b^{\delta}_\eps(\cos\theta)}
\newcommand{\beps}{b_\eps}

\newcommand{\hu}{\hat{u}}
\newcommand{\hj}{\hat{j}}
\newcommand{\hk}{\hat{k}}
\newcommand{\vs}{v_*}
\newcommand{\fs}{f_*}
\newcommand{\feps}{f_\eps}

\newcommand{\geps}{g_\eps}
\newcommand{\qgeps}{Q_{\geps}}

\newcommand{\meps}{m_\eps }

\newcommand{\dR}{\mathrm{d}R}

\newcommand{\Bbeps}{B_\eps}
\newcommand{\p}{\mathcal{P}}
\newcommand{\peps}{\p_\eps}

\newcommand{\Beps}{\mathcal{B}_\eps}

\newcommand{\etaps}{\eta_p^*}
\newcommand{\psiqs}{\psi_q^*}

\newcommand{\psinfs}{\psi_\infty^*}
\newcommand{\phipps}{\phi_{p'}^*}

\newcommand{\fo}{f_0}

\DeclareMathOperator*{\esssup}{ess\,sup}

\numberwithin{equation}{section}

\date{}
\allowdisplaybreaks

\begin{document}

\title{Existence and uniqueness of $L^p$ solutions to the Boltzmann equation with an angle-potential concentrated collision kernel}

\author{Sona Akopian\\
sakopian@math.utexas.edu,\\
Irene M. Gamba\\
gamba@math.utexas.edu.\\
Department of Mathematics \\
University of Texas at Austin\\
}

\maketitle
{\begin{abstract} 
%
	
We solve the Cauchy problem associated to the space homogeneous Boltzmann equation with an angle-potential singular concentration modeling the collision kernel, proposed in \cite{bp}.	The potential under consideration ranges from Coulomb to hard spheres cases. However, the motivation of such a collision kernel is to treat the case of Coulomb potentials, on which this particular form of collision operator is well defined. We also show that the scaled angle-potential singular concentration in a grazing collisions limit makes the Boltzmann operator converge in the sense of distributions to the Landau operator acting on the Boltzmann solutions.

\end{abstract}}

{\bf Keywords:} {kinetic theory, soft potentials, Coulomb forces, grazing collisions limit, long range interactions, abstract ODE theory.\\ 

\tableofcontents

\section{Introduction} \label{introduction}
Aside from being evolution equations in nonlocal (differential-integral) form, the Boltzmann and Landau equations are closely related mathematically in the sense that the Landau equation was formally derived in 1937 (see ~\cite{landau}) from the Boltzmann equation which, on its own, cannot describe plasmas. This is due to the fact that the intermolecular Coulomb forces are so strong, that the singularities they create in the collision integral of the Boltzmann equation (the nonlocal integral term) are of the critical order at which the integral diverges.
Hence, the Boltzmann equation in this case is ill - posed.
Landau, however, was able to use the structure of the Boltzmann equation's \textit{collision kernel} (the weight in the collision integral that models probability rates of two interacting particles transitioning from their pre- to their post-collisional states) heuristically to derive a proper, convergent, collision operator that describes particle interactions in this special case.


\subsection{Description of the Boltzmann and Landau equations}\label{description of equations}
The formulation of the problem, in the $x$-uniform framework (known as the space homogeneous problem), is posed as follows:
let $f=f(v,t)$, for $(v,t)\in \R^3\times (0,\infty)$, be a probability density function describing the probability of finding a particle with velocity $v$ at time $t$. Let $\vs$ denote the velocity of a particle about to collide with the first, and let $v'$ and $\vs'$ denote their respective velocities before or after a reversible (elastic) collision. Also in the elastic case, collisions must conserve momentum and kinetic energy:
	\begin{align}
		& v' + \vs' = v + \vs 
		&&
		\text{ and } 
		& |v'|^2 + |\vs'|^2 = |v|^2 + |\vs|^2.
	\label{collisionconslaws}\end{align}
Defining the relative velocities $u:= v-\vs$ and $u': = v'-\vs'$ and letting $\sigma: = \hu' = u'/|u'| \in S^2$ denote the \textit{scattering direction,} the post - (or pre-) collisional velocities may be written as
	\begin{align*}
		& v' = v'(v,\vs,\sigma) = v + \frac{1}{2}(|u|\sigma - u),
		& \vs' = \vs'(\vs,v,\sigma) = \vs - \frac{1}{2}(|u|\sigma - u).
	\end{align*}
 One can represent $\sigma \in S^2$ as
 	\begin{equation}
		\sigma = \hu \cos\theta + \omega\sin\theta,
	\label{sigma}\end{equation}
where $\omega \in u^\perp,$ $|\omega| = 1$ is in turn decomposed into
	\begin{equation}
		\omega = \hj \cos\phi + \hk\sin\phi.
	\label{omega}\end{equation}
Here $j,k \in u^\perp$ are defined as $ j:= (1,0,0) - \hu\hu_1$, $k = j\times\hu.$

The Cauchy problem for the Boltzmann equation is written in strong form as 
	{\small
	\begin{equation}\label{bte}
		\begin{dcases}
			&\partial_t f(v)
			= Q_{B}(f,f)(v)
			:= \iint_{\R^3\times S^2}B(|u|^\gamma, \hu\cdot\sigma)\cdot\\
			& \hspace{2.25in} \cdot \Big(f\left(v'\right)f\left(\vs'\right) - f(v)f(\vs)\Big)\dsigma\dvs\\
			& f(v)\Big|_{t=0} = \fo(v),
		\end{dcases}
	\end{equation}
	}
where $\hat{z}: = z/|z|$ for any $z\in \R^3,$ and we have omitted the $t$ variable for convenience.
The collision kernel, $B(|u|^\gamma,\hu\cdot\sigma)$, is often modeled as 
	\begin{multline}
		B(|u|^\gamma, \hu\cdot\sigma)\dsigma 
		:= |u|^\gamma b(\hu\cdot\sigma)\dsigma\\
		= |u|^\gamma b(\cos\theta)\sin\theta\dtheta\domega
		= |u|^\gamma \sin^{-m}(\theta/2)\sin\theta\dtheta\domega,
	\label{B}\end{multline}
and the spaces $L^p_k$ are defined as 	
	\begin{equation*}
		L^p_k(\R^d): = \left\{ f\in L^p(\R^d): \|f\|_{L^p_k(\R^d)} = \left(\int_{\R^d} f^p(v) (1+ |v|^2)^\frac{pk}{2}\dv\right)^{\frac{1}{p}}\right\}.
	\end{equation*}

The parameter $\gamma \in (-3, 0)$ corresponds to soft potentials (repulsive forces), and $\gamma = -3$, which is only possible in (\ref{lte}), corresponds to Coulomb forces.  The case $\gamma \in (0, 1]$ corresponds to hard potential and was extensively studied in \cite{dvvillani1, dvvillani2, dvvillani3}, and $\gamma = 0$ describes Maxwell molecule interactions (see for example ~\cite{alonsocarneirogamba, ricardo, bgp, bp, villani}).

The weight $b(\hu\cdot\sigma)\dsigma$ is known as the \textit{angular cross section}, and while it does not need to be defined exactly as it is in (\ref{B}), $b(\cdot)$ is always an even, nonnegative function that must satisfy
	\begin{equation}
		\int_0^{\pi/2}b(\cos\theta)\sin^2\theta\cos(\theta/2)\dtheta < \infty
	\label{bsingularity}\end{equation}
in order for $Q_B$ to be well defined in weak form (see \cite{advw, ccl} for discussion on the cancellation lemma). In the case of (\ref{B}), this means that $m < 4.$
 
In the case of Coulomb potentials (when $\gamma = -3$), the cross section has been determined to be of the \textit{Rutherford type}, where $b(\hu\cdot\sigma)\dsigma = b(\cos\theta)\sin\theta\dtheta \sim \sin^{-3}(\theta/2)\dtheta$ for $\theta <<\pi/2.$ This corresponds to $m=4$ from (\ref{B}), so $Q_B$ is no longer well defined. Recall that in this case we use the Landau equation to model particle interactions. Nonetheless, an $\eps$-truncation of the Boltzmann equation's $b(\cos\theta)$ helps us analyze the asymptotics and derive the Landau equation, whose strong form is
	{\small
	\begin{equation}
		\begin{dcases}
			&  \partial_tf(v) = Q_L(f,f)(v) 
			: = \nablav \cdot \int_{\R^3} |u|^{\gamma + 2}\Pi(u) \cdot\\
			& \hspace{2.25 in} \cdot \Big(  f(\vs)\nabla f(v)- \nabla f(\vs) f(v)\Big)\dvs,\\
			 &  f(v)\Big|_{t=0} \!\!\! =: \fo(v),
		\label{lte}\end{dcases}
	\end{equation}
	}
where $\Pi(u): = I_{3\times 3} - \hu\otimes\hu \in \R^{3\times 3}$ projects onto the space $u^\perp.$ 

There are several important similarities between the Boltzmann and Landau equations.
For example, one can check, by using (\ref{collisionconslaws}), the symmetry of $B(|u|,\hu\cdot\sigma)$ and exchanging variables in $Q_B(f,f)$ and $Q_L(f,f)$, that solutions of both the Boltzmann and Landau equations conserve mass, momentum and kinetic energy:
\begin{equation}
		\frac{d}{dt}\int_{\R^3} f(v,t)(1, v, |v|^2)\dv
		= \int Q_{B,L} (f,f)(v,t) (1,v, |v|^2)\dv = 0,
	\end{equation}
that is,	
	\begin{equation}
		\int f(v,t) (1,v,|v|^2)\dv = \int \fo(v)(1,v,|v|^2)\dv.
	\label{conslaws}\end{equation}
Also both equations satisfy the H-Theorem, meaning that
	\begin{equation}
		-\frac{d}{dt}\mathcal{H}(t)
		= - \frac{d}{\dt}\int f\log f\dv
		= \int Q_{B,L}(f,f)\log f\dv
		\leq 0,
	\label{hthm}\end{equation}
with equality holding if and only if $f$ is a Gaussian in the velocity variable. 

However, in order to see exactly how $Q_B$ turns into $Q_L$ when collisions become grazing, we need to carefully study the limiting behavior of a properly $\eps-$truncated collision cross section, $\beps$, for which the Boltzmann equation does not yet fall apart.

	
\subsection{The grazing collisions limit}

The most common truncation of the Boltzmann collision cross section is
	\begin{equation}
		\beps(\cos\theta) = \frac{I_{\theta\geq\eps}}{|\log\sin(\eps/2)|}b(\cos\theta)
	\label{ruthtrunc}\end{equation}
(see ~\cite{landau, dl, villani,gambahaack}). 
In ~\cite{gambahaack}, the authors were able to extend this to an even stronger $\theta$-singularity in $b(\cos\theta)$ with a suitable truncation, and still obtain the Landau equation: for $\delta \in [0,2)$ they define 
	\begin{equation}
		\bde: = \frac{I_{\theta\geq\eps}}{H_\delta(\sin(\eps/2))}\sin^{-(4+\delta)}(\theta/2),
	\label{ghtruncation}\end{equation}
where $H_\delta $ is such that $H_\delta'(x) = x^{-(\delta+1)}$ (if $\delta = 0$ then we recover (\ref{ruthtrunc})). Moreover, the rate of convergence of the Boltzmann collision integral $Q_{B_\eps}$ to the corresponding Landau collision term is much higher for $\delta >0.$ In a sense, the angular cross sections $\beps$ approximate a singular point mass distribution as $\eps\longrightarrow 0$, which is a signature of the Landau derivation. This limit is called the \textit{grazing collisions limit.}

It's important to note that one does not need to use the exact truncation (or the exact collision kernel) above to get the grazing collisions limit. In fact, according to \cite{villani} it suffices for $\Bbeps$ to satisfy the following three conditions, pointwise in $u\in\R^3$:
	\begin{align}
		& \beta_2[\Bbeps](u)
		:= \int_0^{\frac{\pi}{2}}\Bbeps(|u|^\gamma, \cos\theta)\sin^2(\theta/2)\sin\theta\dtheta \nonumber\\
		& \hspace{2.75in} \to \frac{2}{\pi}|u|^\gamma \text{ as } \eps\to 0,\label{L1}\\
		& \forall k>2,  \hspace{.1in}
		\beta_k[\Bbeps] (u):= \int_0^{\frac{\pi}{2}}\Bbeps(|u|^\gamma, \cos\theta) \sin^{k}(\theta/2)\sin\theta\dtheta \nonumber\\
		&\hspace{3in}\to 0  \text{ as } \eps\to 0, \label{L2}\\
		&
		|u|^{-\gamma}\Bbeps(|u|^\gamma, \cos\theta) 
		\to 0 \text{ as } \eps\to 0, \nonumber\\
		&\hspace{1.75in} \text{ uniformly on } \{\theta>\theta_0\}, \hspace{.1cm} \forall \theta_0 \in (0, \pi/2). \label{L3}
	\end{align}
Indeed, one can show that if (\ref{L1})- (\ref{L3}) hold and $B_\eps(|u|^\gamma, \hu\cdot\sigma) = |u|^\gamma \beps(\hu\cdot\sigma)$, then $Q_{B_\eps}(f,f) \rightarrow Q_L(f,f)$ as $\eps\rightarrow 0$ in the sense of distributions. A sketch of the following theorem can be found in \cite{villani}:
%
%
\newtheorem{qconvergence}{Proposition}[section]
\begin{qconvergence}
Consider a sequence of nonnegative collision kernels, $B_\eps = B_\eps(|u|^\gamma, \hu\cdot\sigma) = |u|^{\gamma}b_\eps(\hu\cdot\theta)$, $ -3\leq \gamma \leq -1,$ satisfying properties (\ref{L1}) - (\ref{L3}), and let $0 \leq f\in L^1_2\cap L^p$, where
	\begin{align*}
		& p \geq \frac{6}{6-|\gamma + 2|} \text{ if } \gamma \leq -2,\\
		& p > 1 \text{ if } \gamma > -2.
	\end{align*}
Then for any test function $\vphi \in C_0^\infty(\R^3)$ and for any $t >0,$
	\begin{equation}
		\lim_{\eps\to 0}\left|\int \left(Q_{B_\eps}(f,f)(v,t)- Q_L(f,f)(v,t)\right)\vphi(v)\dv\right|
		= 0.
	\end{equation}
\label{qconvergence}\end{qconvergence}
\begin{proof}
We will first formally split $\int Q_{B_\eps}(f,f)\vphi\dv$ and $\int Q_L(f,f)\vphi\dv$ into several integrals, and justify the splitting at the end.
By an exchange of variables, one can check that formally,
	\begin{multline}
		\int Q_L(f,f)(v)\vphi(v)\dv\\
		= \iint f\fs |u|^\gamma \left(-2(\nabla\vphi - \nablas\vphis)\cdot u 
									+ \frac{1}{2}  |u|^2 (D^2\vphi + D_*^2\vphis):\Pi(u)\right)\dvs\dv\\
		= \frac{1}{2} \iint f\fs |u|^\gamma G_L(v,\vs)\dvs\dv,
	\end{multline}
where 
	\begin{multline}
		G_L(v,\vs)
		: = -2(\nabla\vphi - \nablas\vphis)\cdot u 
		+  \frac{1}{2}|u|^2 (D^2\vphi + D_*^2\vphis):\Pi(u).
	\end{multline}
Define
	\begin{align*}
		&G_L^1(v,\vs) = -4(\nabla\vphi - \nablas\vphis)\cdot u = -4(\partialvi\vphi - \partialvsi\vphis)u_i,\\
		& G_L^2(v,\vs) = |u|^2 (D^2\vphi + D_*^2\vphis):\Pi(u) = |u|^2(\partialvivj\vphi + \partialvsivsj\vphis)\Pi(u)_{ij}.
	\end{align*}
Similarly for the Boltzmann collision term,
	\begin{multline*}
		\int Q_{B_\eps}(f,f)(v)\vphi(v)\dv
		= \frac{1}{2}\iint f\fs |u|^{\gamma} \int_{S^2}b_\eps(\hu\cdot\sigma) (\vphi' + \vphis' - \vphi - \vphis)\dsigma\dvs\dv\\
		= \frac{1}{2}\iint f\fs G[\Bbeps](v,\vs)\dvs\dv,
	\end{multline*}
where
\begin{multline*}
		G[\Bbeps](v,\vs)
		:=  |u|^\gamma \int_0^{\pi/2}B_\eps(|u|^\gamma, \cos\theta)\int_{-\pi}^\pi  (\vphi' + \vphis' - \vphi - \vphis)\dphi\sin\theta\dtheta\\
		= \int_{S^2}B_\eps(|u|^\gamma, \hu\cdot\sigma)(\vphi' + \vphis' - \vphi - \vphis)\dsigma.
	\end{multline*}
We begin by taking the second order Taylor expansion of $\vphi' + \vphis' - \vphi -\vphis$, keeping in mind that $\vs'-\vs = -(v'-v)$:
	\begin{multline}
		(\vphi' - \vphi) + (\vphis' - \vphis) =\\ 
		= \nabla\vphi(v)\cdot(v'-v) 
		+ \frac{1}{2}\partialvivj\vphi(v)(v_i'-v_i)(v_j'-v_j)\\
		+ \frac{1}{6}\partial_{v_i v_j v_k}\vphi(\xi)(v_i'-v_i)(v_j'-v_j)(v_k'-v_k)\\
		+ \nabla\vphi(\vs)\cdot(\vs' - \vs)
		+ \frac{1}{2}\partialvivj\vphi(\vs)(v_i'-v_i)(v_j'-v_j)\\
		+ \frac{1}{6}\partial_{v_iv_jv_k}\vphi(\zeta)(v_i'-v_i)(v_j'-v_j)(v_k'-v_k)\\
		= (\nabla\vphi(v) - \nabla\vphis(\vs))\cdot(v' - v)\\
		+ \frac{1}{2}(\partialvivj\vphi(v) + \partialvsivsj\vphi(\vs))(v_i'-v_i)(v_j'-v_j)\\
		+ \frac{1}{6}(\partial_{\vi\vj\vk}\vphi(\xi)- \partial_{\vsi\vsj\vsk}\vphi(\zeta))(v_i'-v_i)(v_j'-v_j)(v_k'-v_k),
	\label{vphiexpansion}\end{multline}
where $\xi$ and $\zeta$ are convex combinations of $v,v'$ and $\vs,\vs'$ respectively.
Next, substitute this expansion into $G[\Bbeps]$:
	\begin{multline}
		G[\Bbeps](v,\vs)
		=  (\nabla\vphi(v) - \nabla\vphi(\vs))\cdot \int_0^{\pi/2} B_\eps(|u|^\gamma, \cos\theta)\int_{-\pi}^\pi(v'-v)\dphi\sin\theta\dtheta\\
		+ \frac{1}{2} (\partialvivj\vphi + \partialvsivsj\vphis)  \int_0^{\pi/2}B_\eps(|u|^\gamma, \cos\theta) \int_{-\pi}^\pi(\vi'-\vi)(\vj'-\vj)\dphi \sin\theta\dtheta\\
		+ \frac{1}{6} \int_0^{\pi/2} B_\eps(|u|^\gamma, \hu\cdot\sigma)\int_{-\pi}^\pi(\partial_{\vi\vj\vk}\vphi(\xi) - \partial_{\vsi\vsj\vsk}\vphi(\zeta))\cdot\\
		\cdot(\vi'-\vi)(\vj'-\vj)(\vk'-\vk)\dsigma
	\end{multline}
Let
	\begin{align*}
		& G_1[\Bbeps](v,\vs) : =  \int_0^{\pi/2} \Bbeps(|u|^\gamma, \hu\cdot\sigma)\int_{-\pi}^\pi (v'-v)\dphi\sin\theta\dtheta,\\
		& G_2[\Bbeps] (v,\vs) : = \frac{1}{2}\int_0^{\pi/2} \Bbeps(|u|^\gamma, \cos\theta)\int_{-\pi}^\pi (\vi'-\vi)(\vj-\vj)\dphi\sin\theta\dtheta,\\
		& G_3[\Bbeps] (v,\vs) : = \frac{1}{6} \int_0^{\pi/2}\Bbeps(|u|^\gamma, \cos\theta) \int_{-\pi}^\pi(\partial_{\vi\vj\vk}\vphi(\xi)- \partial_{\vsi\vsj\vsk}\vphi(\zeta))\cdot\\
		&\hspace {2.5in} \cdot (v_i'-v_i)(v_j'-v_j)(v_k'-v_k)\sin\theta\dtheta\dphi\\
		& \leq \frac{1}{3}\|D^3\|_{L^\infty} \int_0^{\pi/2} \Bbeps(|u|^\gamma, \cos\theta) \int_{-\pi}^\pi |v'-v|^3\dphi\sin\theta\dtheta.
	\end{align*}
Using the representations of the post-collisional velocities and of the scattering direction from Section \ref{description of equations}, it is not hard to show that
	\begin{align}
		 & \int_{-\pi}^\pi (v'-v)\dphi
		= \pi u(\cos\theta -1)
		= -2\pi u \sin^2(\theta/2), \\
		&\int_{-\pi}^\pi (\vi' - \vi)(\vj'-\vj)\dphi 
		= \frac{\pi}{2}(\cos\theta -1)^2u_iu_j
		+ \frac{\pi}{4}\Pi(u)_{ij}|u|^2\sin^2\theta \nonumber \\
		& \hspace{.5in} = \pi\sin^4(\theta/2)(2u_iu_j - |u|^2\Pi(u)_{ij})
		+ \pi |u|^2\Pi(u)_{ij} \sin^2(\theta/2),\\
		& \int_{-\pi}^\pi |v'-v|^3\dphi
		= |u|^3 \int_{-\pi}^\pi \sin^3(\theta/2)\dphi
		= 2\pi \sin^3(\theta/2).
	\end{align}
Then $G_1[\Bbeps]$ and $G_2[\Bbeps]$ become

	\begin{multline}
		G_1[\Bbeps](v,\vs)
		= -2\pi u \int_{0}^{\pi/2} B_\eps(|u|^\gamma, \cos\theta)\sin^2(\theta/2)\sin\theta\dtheta\\
		= -2\pi  u \beta_2[\Bbeps](u)
	\label{G1}\end{multline}
and
	\begin{multline}
		G_2[\Bbeps](v,\vs)
		= \frac{\pi}{2} (2u_iu_j - |u|^2\Pi(u)_{ij})\int_0^{\pi/2}B_\eps(|u|^\gamma, \cos\theta) \sin^4(\theta/2)\sin\theta\dtheta\\
		+ \frac{\pi}{2} |u|^2\Pi(u)_{ij} \int_0^{\pi/2} B_\eps(|u|^\gamma, \cos\theta)\sin^2(\theta/2)\sin\theta\dtheta\\
		=: \frac{\pi}{2} (2u_iu_j - |u|^2\Pi(u)_{ij}) \beta_4[\Bbeps](u)
		+ \frac{\pi}{2}  |u|^2 \Pi(u)_{ij} \beta_2[\Bbeps](u),
	\label{G2}\end{multline}
and $G_3[\Bbeps]$ is bounded by
	\begin{multline}
		G_3[\Bbeps](v,\vs)
		\leq  \frac{2\pi}{3}|u|^{3} \|D^4\vphi\|_{L^\infty}\int_0^{\pi/2} \Bbeps(|u|^\gamma, \cos\theta)\sin^3(\theta/2)\sin\theta\dtheta\\
		=  \frac{2\pi}{3}|u|^{3} \|D^4\vphi\|_{L^\infty} \beta_3[\Bbeps](u).
	\end{multline}
Together,
	\begin{multline}
		\int Q_{B_\eps} (f,f)(v)\vphi(v)\dv
		= \frac{1}{2} \iint f\fs ((\nabla\vphi - \nablas\vphis)\cdot G_1[\Bbeps] (v,\vs)\dvs\dv\\
		+ \frac{1}{2} \iint f\fs (\partialvivj\vphi + \partial\vsi\vsj\vphis)G_2[\Bbeps](v,\vs)\dvs\dv\\
		+ \frac{1}{2} \iint f\fs G_3[\Bbeps](v,\vs)\dvs\dv\\
		= -\pi \iint \beta_2[\Bbeps] (u) f\fs(\nabla\vphi - \nablas\vphis)\cdot u\dvs\dv\\
		+ \frac{\pi}{4}  \iint \beta_4[\Bbeps](u) f\fs (\partialvivj\vphi + \partial\vsi\vsj\vphis) (2u_iu_j - |u|^2 \Pi(u)_{ij}) \dvs\dv\\
		+ \frac{\pi}{4} \iint \beta_2[\Bbeps](u) f\fs |u|^{2}  (\partialvivj\vphi + \partial\vsi\vsj\vphis)\Pi(u)_{ij}\dvs\dv\\
		+ \frac{1}{2}\iint f\fs G_3[\Bbeps](v,\vs)\dvs\dv.
	\label{alltogether}\end{multline}	
Now, we show that the four integrals at the end of (\ref{alltogether}) are bounded. This will justify this splitting of $\int Q_{B_\eps}(f,f)\vphi$ and $\int Q_L(f,f)\vphi.$ 
First, the structure of $\Bbeps(|u|^\gamma, \hu\cdot\sigma) = |u|^\gamma, \beps(\hu\cdot\sigma)$ and assumptions (\ref{L1}), (\ref{L2}) imply that $\beta_k[\Bbeps](u) \leq A_k|u|^\gamma$ for some $K>0.$ Using this, and that $\vphi \in C_0^\infty,$
	\begin{equation}
		\iint f\fs \beta_2[\Bbeps] (u) (\nabla\vphi - \nablas\vphis)\cdot u\dvs\dv
		 \leq A_2 \|D^2\vphi\|_{L^\infty} \iint f\fs |u|^{\gamma + 2}\dvs\dv,
	\end{equation}
	\begin{multline}
		\iint\beta_4[\Bbeps](u)  f\fs (\partialvivj\vphi + \partialvsivsj\vphis)(2u_iu_j - |u|^2\Pi(u)_{ij})\dvs\dv \\
		\leq 6A_4 \|D^2\vphi\|_{L^\infty} \iint f\fs |u|^{\gamma+2}\dvs\dv,
	\end{multline}
	\begin{multline}
		 \iint\beta_2[\Bbeps] (u)   f\fs |u|^{ 2} \dvs\dv (\partialvivj\vphi + \partialvsivsj\vphi)\Pi(u)_{ij}\dvs\dv  \\
		 \hspace{2in} \leq 2A_2\|D^2\vphi\|_{L^\infty}\iint f\fs |u|^{\gamma + 2}\dvs\dv,
	\end{multline}
and
	\begin{equation}
		\iint f\fs G_3[\Bbeps]\dvs\dv
		\leq \frac{2\pi}{3}A_3\iint f\fs |u|^{\gamma+3}\dvs\dv.
	\label{G3}\end{equation}

It remains to show that $\iint f\fs |u|^{\gamma+2}, \iint f\fs |u|^{\gamma + 3}$ are finite. For this, we apply the following lemma, which was inspired by Lemma 4 from \cite{dv}:
\newtheorem{dvlemma4'}[qconvergence]{Lemma}
\begin{dvlemma4'}
Let $p>1$, $k>0$ and $\alpha \leq k$ such that $\alpha p' > -6.$ If $f\in L^1_k \cap L^p(\R^3)$ and $h(v): = |v|^{\alpha},$ then there exists $C = C(p)>0$ such that
	\begin{align}
		& \iint f(v)f(\vs)|v-\vs|^\alpha \dvs\dv
		\leq 2^k\|f\|_{L^1(\R^3)}\|f\|_{L^1_k(\R^3)} \text{ if } \alpha \geq 0,\\
		& \iint f(v)f(\vs)|v-\vs|^\alpha \dvs\dv
		\leq \|f\|_{L^1(\R^3)}^2 + C\|f\|_{L^p(\R^3)}^2  \text{ if } \alpha <0.
	\end{align}
\label{dvlemma4'}\end{dvlemma4'}
\begin{proof}
If $\alpha \geq 0,$ then we can use the convexity of $x\mapsto x^\alpha$ to get
	\begin{equation*}
		|v-\vs|^\alpha 
		\leq 2^{\alpha -1} (|v|^\alpha + |\vs|^\alpha)
		\leq 2^{k -1 } \left((1+|v|^2)^{\frac{\alpha}{2}} + (1+ |\vs|^2)^{\frac{\alpha}{2}}\right), 
	\end{equation*}
so
	\begin{equation*}
		\iint f(v) f(\vs) |v-\vs|^\alpha \dvs\dv
		\leq 2^k\|f\|_{L^1}\|f\|_{L^1_k}.
	\end{equation*}
If $\alpha < 0,$ let $h_1(v): = |v|^{\alpha}\1_{|v|\leq 1}$ and $h_2(v): = |v|^\alpha\1_{|v|>1},$ so that $h_1 + h_2 = h.$
Then
	\begin{equation*}
		\iint f(v)f(\vs)h(v-\vs)\dvs\dv
		= \int f(v) f\ast h_1(v) \dv
		+ \int f(v) f\ast h_2(v)\dv.
	\end{equation*}
By Holder's and Young's inequalities,
	\begin{align}
		& \int f f\ast h_1\dv
		\leq \|f\|_{L^p} \|f\ast h_1\|_{L^{p'}}
		\leq \|f\|_{L^p}^2 \|h_1\|_{L^{p'/2}}, \label{h1}\\
		& \int f f\ast h_2 \dv
		\leq \|f\|_{L^1}\|f\ast h_2\|_{L^\infty}
		\leq \|f\|_{L^1}^2 \|h_2\|_{L^\infty}
		= \|f\|_{L^1}^2 \label{h2}
	\end{align}	
Note that $\|h_1\|_{L^{p'/2}}$ depends only on $|S^2|$ and $p'$.
\end{proof}

Now that our steps until now have been justified, we can take the limit as $\eps\to 0$ in (\ref{alltogether}). (\ref{L1}) - (\ref{L3}) allow the second and fourth integrals to vanish, leaving us with
	\begin{multline}
		\int Q_{\Bbeps}(f,f)(v)\vphi(v)\dv
		\to
		-2\iint |u|^\gamma f\fs (\nabla\vphi- \nablas\vphis)\cdot u\dvs\dv\\
		+ \frac{1}{2} \iint |u|^\gamma f\fs (D^2\vphi + D_*^2\vphis):\Pi(u)\dvs\dv\\
		= \int Q_L(f,f)(v)\vphi(v)\dv.
	\end{multline}
\end{proof}

If, additionally,  $\feps$ is a solution to (\ref{bte}) with $B= B_\eps,$ then one can go a step further and replace $f$ in Proposition \ref{qconvergence} with $\feps:$

\newtheorem{qsolconvergence}[qconvergence]{Theorem}.
\begin{qsolconvergence}
Let $\feps \in L^p(\R^3)\cap L^1_2(\R^3),$ $p > \frac{6}{5},$ be a weak solution of (\ref{bte}) with the collision cross section $\beps$ satisfying (\ref{L1}) - (\ref{L3}), and with $0\leq \fo \in L^1_2\cap L\log L(\R^3)$. Then, for all $t>0,$ $Q_{B_\eps}(f,f) \to Q_L(f,f)$ as distributions. That is, for any $\vphi \in C_0^\infty,$
	\begin{equation}
		\lim_{\eps \longrightarrow 0}\left|\int_{\R^3} \left(\qgeps(\feps, \feps)(v,t) - Q_L(\feps, \feps)(v,t)\right)\vphi(v)\dv\right|
		= 0.
	\end{equation}
\label{qsolconvergence}
\end{qsolconvergence}
The proof of this theorem is almost identical to the proof the previous proposition.

	
\subsection{Weak and weak-H formulations.}
	
Let $\vphi = \vphi(v,t) \in C^1(\R^+, C_0^\infty(\R^3))$ and consider the weak form of $Q_B$ for a collision kernel $B(|u|, \hu\cdot\sigma) = |u|^\gamma b(\hu\cdot\sigma),$ with $b$ even, nonnegative and symmetric about $\theta = \pi/2$. Then $B$ is invariant under the change of variables $(v,\vs)\leftrightarrow (\vs,v)$ and $(v,\vs)\leftrightarrow (v',\vs').$ Furthermore, the Jacobian of these transformations has an absolute value of one, therefore we may write the weak form of $Q_B$ as
	\begin{multline}
		\int_0^t\int_{\R^3} Q_B(f,f)(v,s)\vphi(v,s)\dv\ds\\
		= \int_0^t \iint_{\R^3\times\R^3} \int_{S^2} B(|u|^\gamma, \hu\cdot\sigma)(f'\fs'-f\fs) \vphi(v,s)\dsigma\dvs\dv\ds\\
		\hspace{-1.75in} = \frac{1}{4}\int_0^t \iint_{\R^3\times\R^3}\int_{S^2} B(|u|^\gamma, \hu\cdot\sigma)\cdot\\
		\cdot (f'\fs' - f\fs)(\vphi + \vphis - \vphi' - \vphis')\dsigma\dvs\dv\ds.
	\label{weakhqb}\end{multline}
In \cite{villani} the author ensures that the integral on right hand side of (\ref{weakhqb}) converges by making the extra assumption that solutions of the Boltzmann equation have \textit{finite entropy decay,} that is,
	\begin{multline}
		0 
		\leq -\frac{d}{dt}\mathcal{H}(t)
		= - \frac{d}{dt} \int f(v,t)\log f(v,t)\dv \\
		 \hspace{-1.5in} = \frac{1}{4} \iint_{\R^3\times\R^3} \int_{S^2}B(u,\hu\cdot\sigma)\cdot \left(f'\fs' - f\fs \right)\cdot\\
		\cdot \left(\log f'\fs' - \log f\fs\right)\dsigma\dvs\dv 
		< \infty.
	\label{entropycondition}\end{multline}
Because of the assumption on the entropy, the variational formulation of the Boltzmann equation (\ref{bte}) with (\ref{weakhqb}) representing $\int Q_B\vphi$ is called the \textit{weak-H form}, and its solutions are consequently the weak-H, or H solutions.

However, if the singularities of $B(|u|^\gamma, \hu\cdot\sigma)$ are mild enough - that is, if $\gamma \geq -2$ and if (\ref{bsingularity}) holds - then the right hand side of (\ref{weakhqb}) is well defined even without the assumption (\ref{entropycondition}). In fact, in this case we can even go further by splitting $Q_B$ into its \textit{gain and loss parts,} $Q_B^+$ and $Q_B^-$ (which are still well defined):
	\begin{multline*}
	 Q_B(f,f) = \iint_{\R^3\times S^2} B(|u|^\gamma, \hu\cdot\sigma)(f'\fs' - f\fs)\dsigma\dvs\\
	 = \iint_{\R^3\times S^2} B(|u|^\gamma, \hu\cdot\sigma)f'\fs'\dsigma\dvs
	 - \iint_{\R^3\times S^2} B(|u|^\gamma, \hu\cdot\sigma) f\fs\dsigma\dvs\\
	 =: Q_B^+(f,f) - Q_B^-(f,f).
	\end{multline*}
Then we can break up $\int Q_B\vphi$ into
	\begin{multline}
		\int_0^t \int_{\R^3} Q_B(f,f)(v,s)\vphi(v,s)\dv\ds\\
		= \frac{1}{4}\int_0^t \iint \int_{S^2}B(|u|, \hu\cdot\sigma) f'\fs' (\vphi + \vphis - \vphi' -\vphis')\dsigma\dvs\dv\ds\\
		-\frac{1}{4}\int_0^t \iint_{\R^3\times\R^3} \int_{S^2}B(|u|, \hu\cdot\sigma) f\fs (\vphi + \vphis - \vphi' -\vphis')\dsigma\dvs\dv\ds\\
		= \frac{1}{2}\int_0^t \iint_{\R^3\times\R^3}  f\fs \int_{S^2}B(|u|, \hu\cdot\sigma)(\vphi' + \vphis' - \vphi -\vphis)\dsigma\dvs\dv\ds,
	\label{weakqb}\end{multline}
and the right hand side of (\ref{weakqb}) is still well defined. Because no additional assumption on the entropy is required in this case, this form of $\int Q_B\vphi$ is stronger, and is known simply as the \textit{weak form.} The weak formulation of (\ref{bte}) would then have (\ref{weakqb}) on its left hand side, and solutions to this problem are called \textit{weak solutions.}

The precise definitions of weak and weak-H solutions of the Boltzmann equation are defined by Villani in \cite{villani} as follows:
\newtheorem{defweaksol}[qconvergence]{Definition}
	\begin{defweaksol}
		A function $f(v,t)$ is said to be a weak solution of the Boltzmann equation with initial data $0\leq f_0 \in L^1_2(\R^3)$ if the following conditions are satisfied:
		\begin{enumerate}[(i)]
			\item 
			\begin{align*}
				& f\geq 0, 
				f\in\mathcal{C}(\R^+, \mathcal{D}'),
			 	f\in L^1([0,T], L^1_{2+\gamma}),\\
				& \forall t\geq 0, f(\cdot, t) \in L^1_2(\R^3)\cap L\log L(\R^3),
			\end{align*}
				
			\item
			\begin{equation*}
				f(v,0) = f_0(v) \text{ for a.e. } v\in \R^3
			\end{equation*}
			
			\item 	
			\begin{align}
			&\forall t \geq 0,
				\int f(v,t) \left(\begin{array}{c} 1\\ v \\ |v|^2 \end{array}\right)\dv
				= \int f_0(v) \left(\begin{array}{c} 1\\ v\\ |v|^2 \end{array}\right)\dv \label{conservation}\\
			& \int f(v,t)\log f(v,t)\dv \leq \int f_0(v)\log f_0(v)\dv
			\label{entropy condition}
			\end{align}	
			
			\item $\forall \vphi = \vphi(v,t) \in C^1(\R^+, C_0^\infty(\R^3)), \forall t>0,$
			\begin{multline}
				\int f(v,t)\vphi(v,t)\dv - \int f_0(v)\vphi(v, 0)\dv
				- \int_0^t\int f(v,s)\partial_s\vphi(v,s)\dv\ds\\
				\hspace{-1.5in} = \frac{1}{2}\int_0^t\iint f(v,s)f(\vs,s) |u|^\gamma \cdot\\
				\cdot \int_{S^2}b(\hu\cdot\sigma)\left(\vphi' + \vphis' - \vphi - \vphis\right)\dsigma\dvs\dv.
			\label{weakQB}\end{multline}
		\end{enumerate}
	\end{defweaksol}

\newtheorem{defhsol}[qconvergence]{Definition}
	\begin{defhsol}
		A function $f(v,t)$ is an H-solution of the Boltzmann equation if it satisfies all of the conditions (i)-(iii) above, and item (iv) with the right hand side of (\ref{weakQB}) replaced by (\ref{weakhqb}).
	\end{defhsol}

\newtheorem{entropyremark}[qconvergence]{Remark}
\begin{entropyremark}
The difference between weak solutions and H solutions lies {\bf\textit{only}} in the interpretation of $\int Q_B\vphi.$ The entropy assumption, (\ref{entropycondition}), is only a condition under which (\ref{weakhqb}) is well defined. If there is another way to ensure the finiteness of the right hand side of (\ref{weakhqb}) (for example if $B(|u|^\gamma, \hu\cdot\sigma)$ has a special, non-traditional structure), then (\ref{entropycondition}) is not needed. Whether the entropy assumption is needed or not, solutions to the variational problem are called H-solutions as long as $\int Q_B\vphi$ is defined as in (\ref{weakhqb}), and they are called weak solutions if $\int Q_B\vphi$ is defined as in (\ref{weakqb}).
\label{entropyremark}\end{entropyremark}

A similar definition of weak and weak-H solutions can be derived for the Landau equation. One can check in \cite{villani} that the weak-H form of $\int Q_L\vphi$ is
	\begin{multline}
		\int Q_L(f,f)\vphi(v)\dv
		= - \int \sqrt{f\fs|u|^{\gamma+2}}(\nabla\vphi - \nablas\vphis)^T \Pi(u)\cdot \\
		\cdot \left((\nabla - \nablas)\sqrt{f\fs |u|^{\gamma +2}}\right)\dvs\dv,
	\label{weakhql}\end{multline}
which is well defined given the assumption of finite entropy decay of the solutions:
	\begin{multline}
		0
		\leq -\frac{d}{dt}\int f(t,v) \log f(t,v)\dv\\
		 = \int f\fs |u|^{\gamma+2} \left(\frac{\nabla f}{f} - \frac{\nablas\fs}{\fs}\right)^T\Pi(u)
		\left(\frac{\nabla f}{f} - \frac{\nablas\fs}{\fs}\right)\dvs\dv
		< \infty.
	\label{landauentropycondition}\end{multline}
For $\gamma \geq -2$, similar to $Q_B$ one can split $Q_L$ into two integrals to further expand (\ref{weakhql}):
	\begin{multline}
		\int Q_L(f,f)\vphi\dv
		= -\frac{1}{2}\iint f\fs |u|^{\gamma}(\nabla\vphi - \nablas\vphis)\cdot u\dvs\dv\\
		+ 2\iint f\fs |u|^{\gamma + 2} \Pi(u):(D^2\vphi + D_*^2\vphis)\dvs\dv
	\label{weakql}\end{multline}
Now we define a weak solution for the Landau equation:
\newtheorem{defweaklandausol}[qconvergence]{Definition}
	\begin{defweaklandausol}
		A function $f(v,t)$ is said to be a weak solution of the Landau equation with initial data $0\leq f_0 \in L^1_2(\R^3)$ if the following conditions are satisfied:
		\begin{enumerate}[(i)]
			\item 
			\begin{align*}
				& f\geq 0, 
				f\in\mathcal{C}(\R^+, \mathcal{D}'),
			 	f\in L^1([0,T], L^1_{2+\gamma}),\\
				& \forall t\geq 0, f(\cdot, t) \in L^1_2(\R^3)\cap L\log L(\R^3),
			\end{align*}	
			\item
			\begin{equation*}
				f(v,0) = f_0(v) \text{ for a.e. } v\in \R^3
			\end{equation*}
			
			\item 	
			\begin{align}
			&\forall t \geq 0,
				\int f(v,t) \left(\begin{array}{c} 1\\ v \\ |v|^2 \end{array}\right)\dv
				= \int f_0(v) \left(\begin{array}{c} 1\\ v\\ |v|^2 \end{array}\right)\dv \label{conservation}\\
			& \int f(v,t)\log f(v,t)\dv \leq \int f_0(v)\log f_0(v)\dv
			\label{entropy condition}
			\end{align}	
			
			\item $\forall \vphi = \vphi(v,t) \in C^1(\R^+, C_0^\infty(\R^3)), \forall t>0,$
			\begin{multline}
				\int f(v,t)\vphi(v,t)\dv - \int f_0(v)\vphi(v, 0)\dv
				- \int_0^t\int f(v,s)\partial_s\vphi(v,s)\dv\ds\\
				= \frac{1}{2}\int_0^t\iint f(v,s)f(\vs,s) |u|^{-1}(D^2\vphi(v,s) + D_*^2\vphi(\vs,s)):\Pi(u) \dvs\dv\\
				- 2 \int_0^t \iint f\fs|u|^{-3}(\nabla\vphi(v,s) - \nablas\vphi(\vs, s))\cdot u\dvs\dv.
			\label{weakQB}\end{multline}
		\end{enumerate}
	\end{defweaklandausol}
By now, thanks to Desvillettes in \cite{dv}, we know that H-solutions of the Landau equation for all soft potentials (even if $\gamma = -3$) are weak solutions, so $\int Q_L(f,f)\vphi$ can be defined as in (\ref{weakql}), and (\ref{weakhql}) is not necessary. Desvillettes shows that not only does (\ref{landauentropycondition}) ensure that the right hand side of (\ref{weakhqb}) well defined, but that the right hand side of (\ref{weakqb}) is finite too. However, it is important to note that the weak solutions of Desvillettes for very soft potentials require extra integrability, which is obtained by using (\ref{landauentropycondition}), so the entropy assumption is still needed.\\

	
\section{An angle-potential concentrated collision kernel}

\subsection{Description of the kernel}	

Recall that for soft potentials the traditional Boltzmann collision operator, $B(|u|^\gamma, \hu\cdot\sigma) = |u|^\gamma b(\hu\cdot\sigma)$ has two singularities - one when $u=0$, the other when $\hu\cdot\sigma$ =1. However, classical truncations of the kernel only take care of the singularity in $b(\cos\theta),$ not in $|u|^\gamma,$ so in fact the collision operator $Q_{B_\eps}$ is still a singular integral. This is a problem when looking for weak solutions, partly because it does not allow us to make $L^p$ estimates on $Q_B^+$ and $Q_B^-.$ We are therefore tempted to truncate the collision kernel in a way that also controls its singularity at $u=0,$ while still sending $Q_{B_\eps}(f,f)$ to $Q_L(f,f)$ in the grazing collisions limit.

We present here a collision kernel, originally from \cite{bp}, that links the two singularities existing in $B$:
	\begin{equation}
		\geps(|u|^\gamma, \hu\cdot\sigma) 
		= \geps(|u|^\gamma, \mu)
		= \frac{4}{\pi\eps}\delta_0(1-\mu - \min\{2, \eps|u|^{\gamma}\}),
	\label{geps}\end{equation}
where $\mu: = \cos\theta$.
Notice that, unlike the standard $B(|u|^\gamma, \hu\cdot\sigma)$, this collision kernel does not separate its two variables, $|u|^\gamma$ and $\hu\cdot\sigma.$ 


We can check that, for fixed $u \neq 0$, $\geps$ satisfies properties (\ref{L1}) and (\ref{L2}) from before. Indeed, letting $\meps(x): = \min\{2, \eps x^{\gamma}\}$ and $\mu_{\eps}(x): = 1-\meps(x)$ we have
	\begin{multline}
		\beta_2[\geps] (u)
		= \int_0^{\frac{\pi}{2}} \geps(|u|^\gamma, \cos\theta) \sin^2(\theta/2)\sin\theta\dtheta
		= \frac{1}{2}\int_0^{1} \geps(|u|^\gamma,\mu)(1-\mu)\dmu\\
		= \frac{2}{\pi\eps}(1-\mu_\eps(|u|))\1_{\eps|u|^\gamma \leq 1}
		=  \frac{2}{\pi\eps}\meps(|u|)\1_{\eps|u|^\gamma \leq 1}
		= \frac{2}{\pi}|u|^\gamma \1_{\eps|u|^\gamma \leq 1}.
	\end{multline}
Similarly, for $k>2$ and a fixed $u\neq 0,$
	\begin{multline}
		\beta_k[\geps](u)
		= \int_0^{\frac{\pi}{2}} \geps(|u|^\gamma,\mu) \sin^k(\theta/2)\sin\theta\dtheta\\
		= \int_0^1 \geps(|u|^\gamma, \mu) \left(\frac{1}{2}(1-\mu)\right)^{\frac{k}{2}}\dmu
		= 2^{-\frac{k}{2}}\frac{4}{\pi\eps}(1-\mu_\eps(|u|))^{\frac{k}{2}}\1_{\eps|u|^\gamma \leq 1}\\
		= 2^{-\frac{k}{2}}\frac{4}{\pi\eps}\meps(|u|)^{\frac{k}{2}}\1_{\eps|u|^\gamma \leq 1} 
		= \frac{2^{2-\frac{k}{2}}}{\pi} \eps^{\frac{k}{2} -1}|u|^{\gamma k/2} \1_{\eps|u|^\gamma \leq 1} \\
		\longrightarrow 0
		\text{ as } \eps\longrightarrow 0,
	\end{multline}
thus (\ref{L1}) and (\ref{L2}) are satisfied. And in some sense, $\geps$ satisfies (\ref{L3}) as well, because the mass of $\geps$ is concentrated at just one point which corresponds to $\theta$ being very small and $\eps|u|^\gamma \leq 2.$ 

The Boltzmann operator with this new cross section is written as
	\begin{multline}
		\qgeps(f,f)(v,t)
		= \int_{\R^3} \int_{S^2}(f(v',t)f(\vs',t) - f(v,t)f(\vs,t))\geps(|u|^\gamma,\hu\cdot\sigma)\dsigma\dvs\\
		= \int_{R^3} \int_{S^2} f'\fs' \geps(|u|^\gamma, \hu\cdot\sigma)\dsigma\domega\dvs
		- \frac{8}{\eps}f(v,t),
	\end{multline}
and the corresponding Boltzmann equation is
	\begin{equation}
		\begin{dcases}
			& \partial_t f + \frac{8}{\eps}f = \int_{R^3} \int_{S^2} f'\fs' \geps(|u|^\gamma, \hu\cdot\sigma)\dsigma\dvs\\
			& f(v,0) = \fo(v).
		\end{dcases}
	\label{bteqgeps}\end{equation}
In view of Proposition \ref{qconvergence}, it would be reasonable to expect for there to be a grazing collisions limit, and in fact we will show that this is true.

	
\subsection{Connection with $Q_L$}
\newtheorem{convergence}[qconvergence]{Theorem}
\begin{convergence}
	Let $\feps \in L^p(\R^3),$ $p>\max\{\frac{6}{8+ \gamma}, 1\}$ be a weak solution of (\ref{bte}) with the cross section $\geps$ defined as in (\ref{geps}). Then for all time, $|\qgeps(\feps,\feps) - Q_L(\feps,\feps)| \longrightarrow 0$ in the distributional sense as $\eps\to 0$. That is, for any $\vphi \in C_0^\infty(\R^3),$
		\begin{equation}
			\lim_{\eps \longrightarrow 0}\left|\int_{\R^3} \left(\qgeps(\feps, \feps)(v,t) - Q_L(\feps, \feps)(v,t)\right)\vphi(v)\dv\right|
			= 0.
		\end{equation}
\label{convergence}\end{convergence}

\begin{proof}
We need to compute $G_1[|u|^{-\gamma}\geps], G_2[|u|^{-\gamma}\geps]$ and $G_3[|u|^{-\gamma}\geps].$ Let $\meps(|u|): = \min\{\eps|u|^\gamma, 2\}$ and $\mu_\eps(|u|): = 1- \meps(|u|)$ denote the mass of $\delta_0(1-\mu - \meps(|u|).$ Using integration by parts with the Dirac mass $\delta_0$ and recalling (\ref{G1}), (\ref{G2}), we have:
	\begin{multline}
		G_1[\geps](v,\vs)
		= -2\pi u \int_0^1 \geps(|u|^\gamma, \hu\cdot\sigma)\sin^2(\theta/2)\dmu\\
		= -2\pi u \beta_2[\geps](u),
	\end{multline}
	\begin{multline}
		G_2[\geps] (v,\vs)
		= \frac{1}{2} \pi (2u_iu_j - |u|^2\Pi(u)_{ij})\int_0^{\pi/2}\geps(\cos\theta) \sin^4(\theta/2)\sin\theta\dtheta\\
		+ \frac{1}{2} \pi |u|^{2} \Pi(u)_{ij} \int_0^{\pi/2} \geps(|u|^\gamma, \cos\theta)\sin^2(\theta/2)\sin\theta\dtheta\\
		= \frac{\pi}{2} (2u_iu_j - |u|^2\Pi(u)_{ij})\beta_4[\geps](u)
		+ \frac{\pi}{2} |u|^{2} \Pi(u)_{ij} \beta_2[\geps](u)
	\end{multline}
and
	\begin{multline} 
		G_3[\geps](v,\vs)
		\leq \frac{1}{3}\|D^3\vphi\|_{L^\infty}\int_{-\pi}^\pi \int_0^1 \geps(|u|^\gamma, \mu) |v'-v|^3\dmu\dphi\\
		= \frac{2\pi}{3}|u|^3\|D^3\vphi\|_{L^\infty}\int_0^1 \geps(|u|^\gamma, \mu) \sin^3(\theta/2)\dmu\\
		=  \frac{8}{3\eps}|u|^3\|D^3\vphi\|_{L^\infty} \beta_3[\geps](u).
	\end{multline}

All together, 
	\begin{multline}
		\int \qgeps(f,f)\vphi\dv
		= \frac{1}{2} \iint f\fs (\nabla\vphi - \nablas\vphis)\cdot G_1[\geps](v,\vs)\dvs\dv\\
		+ \frac{1}{2} \iint f\fs (\partialvivj\vphi + \partialvsivsj\vphis)G_2[\geps](v,\vs)\dvs\dv\\
		+ \frac{1}{2} \iint f\fs G_3[\geps](v,\vs)\dvs\dv\\
		= -\pi \iint  f\fs \beta_2[\geps](u) (\nabla\vphi - \nablas\vphis)\cdot u\dvs\dv\\
		+ \frac{\pi}{4} \iint \beta_4[\geps](u) f\fs (\partialvivj\vphi + \partialvsivsj\vphis)(2u_iu_j - |u|^2 \Pi(u)_{ij})\dvs\dv\\
		+ \frac{\pi}{4} \iint \beta_2[\geps](u) f\fs (\partialvivj\vphi + \partialvsivsj\vphis) |u|^2 \Pi(u)_{ij}\dvs\dv\\
		+ \frac {1}{2} \iint f\fs G_3[\geps](v,\vs)\dvs\dv
		=: I_1 + I_2 + I_3 + I_4,
	\end{multline}
with $I_1 - I_4$ defined accordingly.  

Now we show that $I_1 - I_4$ are well defined, and send $\eps \to 0.$
	\begin{multline}
		I_1
		= -\pi \iint  f\fs \beta_2[\geps](u) (\nabla\vphi - \nablas\vphis)\cdot u\dvs\dv\\
		= -2\iint_{|u|^\gamma \leq \eps^{-1}}f\fs |u|^\gamma (\nabla\vphi - \nablas\vphis)\cdot u\dvs\dv\\
		\longrightarrow 
		-2\iint f\fs |u|^\gamma (\nabla\vphi - \nablas\vphis)\cdot u\dvs\dv
		= G_L^1(v,\vs).
	\end{multline}
We already showed that this integral converges in Proposition \ref{qconvergence}.
 For $I_2,$ let $1 < P <  \min\left\{\frac{8p-6}{|\gamma|p}, 2\right\}.$ Then
	\begin{multline}
		I_2
		= \frac{\pi}{4} \iint \beta_4[\geps](u) f\fs (\partialvivj\vphi + \partialvsivsj\vphis)(2u_iu_j - |u|^2 \Pi(u)_{ij})\dvs\dv\\
		\leq \frac{1}{4\eps}\iint_{\eps|u|^\gamma\leq 1} f\fs (\eps|u|^\gamma)^P
		|u|^2(\partialvivj\vphi + \partialvsivsj\vphis)(2 - \Pi(u)_{ij})\dvs\dv\\
		\leq \frac{3}{4}\eps^{P-1}\|D^2\vphi\|_{L^\infty} \iint f\fs |u|^{P\gamma + 2}\dvs\dv
		\to 0
	\end{multline}
as $\eps \to 0.$ Note also that, by the construction of $P$ and by Lemma \ref{dvlemma4'}, $I_2 < \infty.$ 
	\begin{multline}
		I_3
		= \frac{\pi}{4} \iint \beta_2[\geps](u) f\fs (\partialvivj\vphi + \partialvsivsj\vphis) |u|^2\Pi(u)_{ij}\dvs\dv\\
		= \frac{1}{2} \iint_{\eps|u|^\gamma \leq 1} |u|^{\gamma +2} f\fs (D^2\vphi + D_*^2\vphis):\Pi(u)\dvs\dv\\
		\longrightarrow \frac{1}{2} \iint |u|^\gamma f\fs (D^2\vphi + D_*^2\vphis):\Pi(u)\dvs\dv
		= G_L^2(v,\vs),
	\end{multline}
and we already showed this integral is finite in Proposition \ref{qconvergence}.
Finally, 
	\begin{multline}
		I_4
		= \frac {1}{2} \iint f\fs G_3[\geps](v,\vs)\dvs\dv\\
		\leq \frac{1}{6}|u|^3\|D^3\|_{L^\infty} \beta_3[\geps](u)
		\leq \frac{4}{3\eps}2^{2-\frac{3}{2}}\eps^{\frac{3}{2}}\|D^3\vphi\|_{L^\infty}\iint_{\eps|u|^\gamma \leq 1} f\fs |u|^{\frac{3\gamma}{2} +3}\dvs\dv\\
		= \frac{\sqrt\eps}{3}2^{4-\frac{3}{2}}\|D^3\vphi\|_{L^\infty}\iint f\fs |u|^{\frac{3\gamma}{2} + 3}\dvs\dv
		\to 0
	\end{multline}
as $\eps\to 0$, and the integral is convergent by Lemma \ref{dvlemma4'}.
The result follows.

\end{proof}

\section{Estimates on $\qgeps$}

In this section we prove that $\qgeps$ maps from $L^1\cap \L^p$ into itself. This will establish its continuity on these spaces and help us show existence, uniqueness and uniform bounds on solutions to (\ref{bteqgeps}).

\subsection{Auxiliary lemmas}
In this section we prove a few lemmas that are necessary to show continuity of $\qgeps.$ We roughly follow the arguments of Lemmas 3,4 and Theorem 5 of \cite{alonsocarneirogamba}.\\

First, we introduce some notation.
Recall from before that we can split $\qgeps$ into its gain and loss parts:
	\begin{multline}
		\qgeps(f,f)(v,t)
		= \int \int_{S^2} f(v',t)f(\vs', t) \geps(|u|^\gamma, \hu\cdot\sigma)\dsigma\dvs
		- \frac{8}{\eps}f(v,t) \\
		=:\qgeps^+(f,f)(v,t)
		= \qgeps^-(f,f)(v,t).
	\end{multline}
Next, for $\eta,\psi, \in C_B(\R^3),$ define
	\begin{align*}
		&\peps(\eta,\psi)(u) : = \int_{S^2}\eta(u^-)\psi(u^+)\geps(|u|^\gamma, \hu\cdot\sigma)\dsigma\\
		& u^-: = \frac{1}{2}(u - |u|\sigma)\\
		& u^+: = \frac{1}{2}(u + |u|\sigma).
	\end{align*}
Finally, define the following radially symmetric functions: for any $f\in L^p$, $1\leq p\leq\infty,$ let
	\begin{align*}
		f_p^*(u) &:= 
		\left(\int_{R\in SO(3)} |f(Ru)|^p \mathrm{d}R\right)^{\frac{1}{p}} \\
		&= \left(\frac{1}{|S^2|}\int_{\sigma \in {S^2}} |f(|u|\sigma)|^p\dsigma\right)^{\frac{1}{p}}
		\text{ if } 1\leq p < \infty,\\
		f_\infty^*(u) &: = \esssup_{R\in SO(3)}|f(Ru)|
		= \esssup_{\sigma\in S^2}|f(|u|\sigma)|.
	\end{align*}
Such functions satisfy the following properties:
\begin{enumerate}[(i)]
\item
$f_p^*$ is radial, i.e. $f_p^*(u) = f_p^*(x)$ whenever $|u| = |x|.$

\item
If g is radial, then $(fg)_p^*(u) = f_p^*g(u).$

\item
If $\dnu$ is a rotationally invariant measure on $\R^3,$ then
	\begin{equation*}
		\int_{\R^3}|f(u)|^p\dnu(u)
		= \int_{\R^3}|f_p^*(u)|^p\dnu(u),
	\end{equation*}
and in particular $\|f\|_{L^p(\R^3)} = \|f_p^*\|_{L^p(\R^3)}.$
\end{enumerate}
We are now ready to introduce the auxiliary lemmas:

\newtheorem{lemma3}{Lemma}[section]
\begin{lemma3}
Let $\eta,\psi,\phi \in C_0(\R^3)$ and $1/p + 1/q + 1/r = 1,$ with $1\leq p,q,r \leq\infty.$ Then,
	\begin{equation}
		\left|\int_{\R^3}\peps(\eta,\psi)(u)\phi(u)\du\right| 
		\leq \int_{\R^3}\peps(\eta_p^*,\psi_q^*)(u)\phi_r^*(u)\du \label{lemma3ineq}.
	\end{equation}
\label{lemma3}\end{lemma3}

\begin{proof}
This lemma and its proof are almost identical to Lemma 3 of \cite{alonsocarneirogamba}.
For some $R\in SO(3)$ we begin with the changes of variable $u \longrightarrow Ru$ and then $\sigma\longrightarrow R\sigma$ in the left hand side of (\ref{lemma3ineq}):

	\begin{multline}
		\left|\int_{\R^3}\peps(\eta,\psi)(u)\phi(u)\du\right|
		= \left|\int_{\R^3}\peps(\eta,\psi)(Ru)\phi(Ru)\du\right|\\
		= \Big|\int_{\R^3}\int_{S^2}\eta\left(\frac{1}{2}(Ru - |u|R\sigma)\right)\psi\left(\frac{1}{2}(Ru + |u|R\sigma)\right)\cdot\\
		\cdot \geps(|u|^\gamma, R\hu\cdot R\sigma)\dsigma\phi(Ru)\du\Big|\\
		\leq \int_{\R^3}\int_{S^2}|\eta(Ru^-)||\psi(Ru^+)|\geps(|u|^\gamma, \hu\cdot \sigma)\dsigma |\phi(Ru)|\du.
	\label{pestimate1}\end{multline}
	
We can characterize the rotation $R = R_{\bar\theta,\bar\omega},$ where $\bar\theta \in [0,\pi],\bar\omega\in S^1$ are defined such that $R_{\bar\theta,\bar\omega}\hu = \hu \cos\bar\theta + \bar\omega\sin\bar\theta.$
Since $R$ is arbitrary and the left hand side of (\ref{lemma3ineq}) does not depend on $R$ we can take the average over all possible rotations in (\ref{pestimate1}) to get
	\begin{multline}
		\left|\int_{\R^3}\peps(\eta,\psi)(u)\psi(u)\du\right|\\
		\leq \int_{\R^3}\int_{S^2}\left(\int_{R\in SO(3)} |\eta(Ru^-)||\psi(Ru^+) |\phi(Ru)|\mathrm{d}R\right)\geps(|u|^\gamma,\hu\cdot \sigma)\dsigma\du\\
		\leq \int_{\R^3} \int_{S^2} \left(\int_{SO(3)}|\eta(Ru^-)|^p\dR\right)^{\frac{1}{p}}
							\left(\int_{SO(3)}|\psi(Ru^+)|^q\dR\right)^{\frac{1}{q}}\cdot\\
							\cdot \left(\int_{SO(3)}|\vphi(Ru^-)|^r\dR\right)^{\frac{1}{r}}
					\geps(|u|^\gamma, \hu\cdot\sigma)\dsigma\du\\
		= \int_{\R^3}\int_{S^2}\Big(\eta_p^*(u^-) \psi_q^*(u^+)\phi_r^*(u)\Big)\geps(|u|^\gamma,\hu\cdot \sigma)\dsigma\du\\
		= \int_{\R^3}\peps(\eta_p^*,\psi_q^*)(u)\phi_r^*(u)\du,
	\label{pestimate2}\end{multline}
where in the end we used Holder's inequality with the exponents $p,q,r.$ This concludes the proof.
\end{proof}

Now we can take advantage of the fact that $\etaps,\psiqs$ are radial to simplify the expression $\peps(\etaps,\psiqs).$ For any function $f: \R^3 \longrightarrow \R,$ let $\bar{f}: \R^+\longrightarrow \R$ be such that $f(x) = \bar{f}(|x|)$ for all $x\in\R^3.$ Then,
	\begin{multline}
		\peps(\etaps,\psiqs)(u)
		= \int_{S^2}\bar{\etaps}(|u^-|)\bar{\psiqs}(|u|^+)\geps(|u|^\gamma, \hu\cdot \sigma)\dsigma\\
		= 2\pi \int_0^1\bar{\etaps}(a_1(|u|^\gamma,\mu))\bar{\psiqs}(a_2(|u|^\gamma,\mu))\geps(|u|^\gamma,\mu)\dmu,
	\end{multline}
where
	\begin{align*}
		a_1(|u|^\gamma,\mu) &:= |u|\sqrt{\frac{1}{2}(1+\mu)} = |u^+|,\\
		 a_2(|u|^\gamma,\mu) &:= |u|\sqrt{\frac{1}{2}(1-\mu)} = |u^-|.
	\end{align*}
This motivates the introduction of a new, simpler bilinear operator defined over bounded, continuous functions of one variable: for $\eta,\psi \in C_B(\R^+)$ define
	\begin{equation*}
		\Beps(\eta,\psi)(x)
		:= \int_0^1 \eta(a_1(x,\mu))\psi(a_2(x,\mu))\geps(x,\mu)\dmu
	\end{equation*}
We prove the following lemma, which is the equivalent of Lemma 4 from \cite{alonsocarneirogamba}:

\newtheorem{lemma4}[lemma3]{Lemma}
\begin{lemma4}
Let $1\leq p \leq \infty$. Then for any $\eta\in L^p(\R^+, x^2\dx)$ and $\psi\in L^\infty(\R^+),$ 
	\begin{equation}
		\|\Beps(\eta,\psi)\|_{L^p(\R^+, x^2\dx)}
		\leq \frac{8}{\pi\eps} \|\psi\|_{L^\infty(\R^+)} \|\eta\|_{L^p(\R^+, x^2\dx)}.
	\label{bboundLp}\end{equation}
\label{lemma4}\end{lemma4}

\begin{proof}
Let $C_\infty = 8$ and for $p<\infty$ let $C_p:=2^{3+ \frac{1}{2p}}.$ Then $C_p \leq 16$ for all $p \in [1,\infty].$
If we let $\mu_\eps(x): = 1-\meps(x) \in [-1, 1)$ be the dirac mass of $\geps,$ then by definition,
	\begin{multline}
		\Beps(\eta,\psi)(x)
		=  \frac{4}{\pi\eps}\eta(a_1(x,\mu_\eps(x)))\psi(a_2(x,\mu_\eps(x)))\1_{0\leq \mu_\eps(x) \leq 1}\\
		= \frac{4}{\pi\eps}\eta(a_1(x,\mu_\eps(x)))\psi(a_2(x,\mu_\eps(x)))\1_{0\leq \meps(x) \leq 1}\\
		= \frac{4}{\pi\eps}\eta(a_1(x,\mu_\eps(x)))\psi(a_2(x,\mu_\eps(x)))\1_{x\geq \eps^{\frac{1}{3}}}.
	\end{multline}
The case $p = \infty$ is trivial, so we assume that $p \neq \infty.$ Then 
	\begin{multline}
		\|\Beps(\eta,\psi)\|_{L^p(\R^+,x^2\dx)}^p
		= \left(\frac{4}{\pi\eps}\right)^p\int_{\eps^{\frac{1}{3}}}^\infty \eta(a_1(x,\mu_\eps(x)))^p\psi(a_2(x,\mu_\eps(x)))^p x^2\dx\\
		\leq  \left(\frac{4}{\pi\eps}\right)^p \|\psi\|_{L^\infty(\R^+)}^p J_{p,\eps}(\eta),
	\end{multline}
where
	\begin{equation*}
		J_{p,\eps}(\eta):= \int_{\eps^{\frac{1}{3}}}^\infty \eta(a_1(x,\mu_\eps(x)))^p x^2\dx.
	\end{equation*}

We estimate the integral $J_{p,\eps}(\eta)$ by performing the change of variable $a_1(x,\mu_\eps(x)) = a_1(x, \eps x^{-3})\longmapsto x:$ 
	\begin{align}
		& a_1(x,\mu_\eps(x))
		 = x\sqrt{\frac{1}{2}(1+\mu_\eps(x))}
		= x\sqrt{1 - \frac{\eps}{2x^3}}
		= \sqrt{x^2 - \frac{\eps}{2x}}, \nonumber\\
		&a_1'(x,\mu_\eps(x))
		 = \frac{1}{2a_1(x,\mu_\eps(x))}\left(2x + \frac{\eps}{2x^2}\right)
		= \frac{4x^3 + \eps}{4x^2a_1(x,\mu_\eps(x))} \cdot \frac{a_1(x,\mu_\eps(x)}{a_1(x,\mu_\eps(x)}, \nonumber\\
		&x^2\dx
		 = \frac{x^2}{a_1'(x,\mu_\eps(x))}\mathrm{d}a_1(x,\mu_\eps(x))
		= \frac{4x^4}{4x^3 + \eps}\frac{ a_1(x,\mu_\eps(x))^2}{x\sqrt{1- \frac{\eps}{2x^3}}}\mathrm{d}a_1(x,\mu_\eps(x))\nonumber\\
		& \hspace{2.5in}  \leq \frac{a_1^2}{\sqrt{1-\frac{\eps}{2x^3}}}\mathrm{d}a_1
		\leq \sqrt 2 a_1^2\mathrm{d}a_1,
	\end{align}
so
	\begin{multline}
		J_{p,\eps}(\eta)
		= \int_{\eps^{\frac{1}{3}}}^\infty \eta(a_1(x,\mu_\eps(x)))^p x^2\dx\\
		\leq \sqrt 2 \int_0^\infty \eta^p(a_1)a_1^2\mathrm{d}a_1
		= \sqrt2 \|\eta\|_{L^p(\R^+, x^2\dx)}^p, 
	\end{multline}
as was to be shown.
\end{proof}

\newtheorem{thm5}[lemma3]{Lemma}
\begin{thm5}
Let $1\leq p \leq \infty$. The bilinear operator $\peps$ extends to a bounded operator from $L^p(\R^3)\times L^\infty(\R^3)$ to $L^p(\R^3)$, and
	\begin{equation}
		\|\peps(\eta,\psi)\|_{L^p(\R^3)}
		\leq \frac{16}{\eps}\|\psi\|_{L^\infty(\R^3)}\|\eta\|_{L^p(\R^3)}.\label{pbound}
	\end{equation}

\label{thm5}
\end{thm5}

\begin{proof}
Let $\eta\in L^p(\R^3), \psi\in L^\infty(\R^3)$ and $\phi \in L^{p'}(\R^3).$ By Lemma \ref{lemma3} combined with a density argument,
	\begin{equation*}
		\int_{\R^3} \peps(\eta,\psi)(u)\phi(u)\du
		\leq \int_{\R^3}\peps(\etaps,\psinfs)(u)\phipps(u)\du.
	\end{equation*}
Since the functions $\etaps, \psinfs, \phipps$ are radial in $u$, let $\bar\eta_p, \bar\psi_\infty, \bar\phi_{p'}: \R^+ \longmapsto \R$ such that for any $u\in \R^3,$ $\etaps(u) = \bar\eta_p(|u|), \psinfs(u) = \bar\psi_\infty(|u|), \phipps(u) = \bar\phi_{p'}(|u|).$ Then for $p\neq \infty$ $\bar\eta_p \in L^p(\R^+, x^2\dx),\bar\psi_q\in L^q(\R^3, x^2\dx)$ and $\bar\phi_r \in L^r(\R^3, x^2\dx)$ and 
	\begin{multline*}
		\int_{\R^3} \peps(\etaps,\psinfs)(u)\phipps(u)\du
		= 2\pi\int_0^\infty\Beps(\bar\eta_p, \bar\psi_\infty)(x) \bar\phi_{p'}(x)x^2\dx\\
		\leq 2\pi \|\Beps(\bar\eta_p, \bar\psi_\infty)\|_{L^p(\R^+, x^2\dx)}\|\phipps\|_{L^{p'}(\R^3)}\\
		\leq 2\pi \frac{8}{\pi\eps}\|\bar\psi_\infty\|_{L^\infty(\R^+)}\|\bar\eta_p\|_{L^p(\R^+,x^2\dx)}\|\phi\|_{L^{p'}(\R^3)}\\
		= \frac{16}{\eps}\|\psi\|_{L^\infty(\R^3)}\|\eta\|_{L^p(\R^3}\|\phi\|_{L^{p'}(\R^3)}
	\end{multline*}
For $p=\infty$ the estimate above is almost identical (replace $\|\bar\eta_p\|_{L^p(\R^+, x^2\dx)}$ with $\|\bar\eta_\infty\|_{L^\infty(\R^+)}$).
This shows us that $\peps(\eta,\psi)$ is a bounded, real-valued linear operator acting on $L^{p'}$, and therefore belongs to $L^p$ with norm bounded by
	\begin{equation}
		\|\peps(\eta,\psi)\|_{L^p(\R^3)} 
		\leq \frac{16}{\eps}\|\psi\|_{L^\infty(\R^3)}\|\eta\|_{L^p(\R^3)},
	\end{equation} 
as was to be shown.
\end{proof}

\subsection{Continuity of $\qgeps$}
Our first step here is to prove boundedness of $\qgeps^+$:

\newtheorem{thm1}[lemma3]{Theorem}
\begin{thm1}
Let $1\leq p,q,r \leq \infty,$ $1/p + 1/q = 1+ 1/r.$ Then the bilinear operator $\qgeps^{+}$ extends to a bounded operator from $L^p(\R^3)\times L^q(\R^3)$ to $L^r(\R^3),$ and 
	\begin{equation}
		\|\qgeps^+(f,h)\|_{L^r(\R^3)}
		\leq \frac{16}{\eps} \|f\|_{L^p(\R^3)}\|h\|_{L^q(\R^3)}
	\label{qgeps+bound}\end{equation}
with $C_p$ defined as before.
	\label{thm1}\end{thm1}

\begin{proof}
First suppose that $(p,q,r) \neq (1,1,1), (1,\infty,\infty), (\infty,1,\infty).$ Let $\psi \in L^{r'}(\R^3)$ be a test function and define
	\begin{multline}
		K_\psi: = \int_{\R^3} \qgeps^+(f,h)(v)\psi(v)\dv
		= \iint_{\R^3\times\R^3}f(v) h(v-u)\p_\eps(\tau_v\mathcal{R}\psi,1)(u)\du\dv\\
		= \iint_{\R^3\times\R^3}\left(f(v)^{\frac{p}{r}} h(v-u)^{\frac{q}{r}}\right)\left(f(v)^{\frac{p}{q'}}\p_\eps(\tau_v\mathcal{R}\psi,1)(u)^{\frac{r'}{q'}}\right)\cdot\\
		\hspace{2in} \cdot \left(g(v-u)^{\frac{q}{p'}}\p_\eps(\tau_v\mathcal{R}\psi,1)(u)^{\frac{r'}{p'}}\right)\du\dv\\
		\leq K_\psi^1 K_\psi^2 K_\psi^3
	\end{multline}
by Holder's inequality with the exponents $p', q', r$, where
	\begin{align*}
		K_\psi^1
		& := \left(\iint f(v)^ph(v-u)^q\du\dv\right)^{\frac{1}{r}}
		= \|f\|_{L^p}^{\frac{p}{r}}\|h\|_{L^q}^{\frac{q}{r}},\\
		K_\psi^2
		& := \left(\iint f(v)^p \p_\eps(\tau_v\mathcal{R}\psi,1)(u)^{r'}\du\dv\right)^{\frac{1}{q'}}
		= \|f\|_{L^p}^{\frac{p}{q'}} \|\p_\eps(\tau_v\mathcal{R}\psi,1)\|_{L^{r'}}^{\frac{r'}{q'}}\\
		K_\psi^3
		& := \left(\iint h(v-u)^q \p_\eps(\tau_v\mathcal{R}\psi,1)(u)^{r'}\du\dv\right)^{\frac{1}{p'}}
		 = \|h\|_{L^q}^{\frac{q}{p'}} \|\peps(\tau_{-v}\mathcal{R}\psi, 1)\|_{L^{r'}}^{\frac{r'}{p'}}.
	\end{align*}
Then by Lemma (\ref{thm5}), 
	\begin{multline}
		K_\psi
		= \langle \qgeps^+(f,h),\psi\rangle_{L^r(\R^3), L^{r'}(\R^3)}\\
		\leq \|f\|_{L^p(\R^3)}\|h\|_{L^q(\R^3)}\|\p_\eps(\tau_{-v}\mathcal{R}\psi, 1)\|_{L^{r'}(\R^3)}\\
		\leq \frac{16}{\eps}\|f\|_{L^p(\R^3)}\|h\|_{L^q(\R^3)}\|\psi\|_{L^{r'}(\R^3)}.
	\label{qgeps+weakbound}\end{multline}
$\qgeps^+(f,h)$ is therefore a bounded linear operator defined on $L^{r'}(\R^3),$ that is, it lies in the dual space $L^r(\R^3)$ and in particular $\psi = (\qgeps^{+}(f,h))^{r-1} \in L^{r'}(\R^3).$ Substituting this choice of $\psi$ into (\ref{qgeps+weakbound}) we obtain
	\begin{multline}
		\|\qgeps^+(f,h)\|_{L^r(\R^3)}^r
		=  \langle \qgeps^+(f,h), (\qgeps^{+}(f,h))^{r-1}\rangle_{L^r(\R^3), L^{r'}(\R^3)}\\
		\leq \frac{16}{\eps}\|f\|_{L^p(\R^3)}\|h\|_{L^q(\R^3)}\|\qgeps^+(f,h)\|_{L^r(\R^3)}^{r-1},
	\end{multline}
and (\ref{qgeps+bound}) follows (the cases $(p,q,r)= (1,1,1), (1,\infty,\infty), (\infty,1,\infty)$ are similar).
\end{proof}

\newtheorem{riclemma}[lemma3]{Lemma}
\begin{riclemma}
For any $p\in[1,\infty],$ $K>1$ and $0\leq \feps\in L^1_2(\R^3)\cap L^p(\R^3)\cap L\log L(\R^3)$ a weak solution to the Boltzmann equation, the following holds:
if $p < \infty,$
	\begin{multline}
		\|\qgeps^+(\feps,\feps)\|_{L^p(\R^3)}
		\leq K^{\frac{1}{p'}}\frac{16}{\eps}\|\feps\|_{L^p(\R^3)}^{\frac{1}{2}}\\
		+ \frac{16}{\eps\log K}\|f_0\|_{L\log L(\R^3)} \|\feps\|_{L^p(\R^3)}
		\textit{ for } p < \infty,
	\label{riclemmaeq1}\end{multline}
and if $p=\infty,$
	\begin{equation}
		\|\qgeps^+(\feps,\feps)\|_{L^\infty(\R^3)}
		\leq \frac{16K}{\eps}.
		+ \frac{16}{\eps\log K}\|f_0\|_{L\log L(\R^3)} \|\feps\|_{L^\infty(\R^3)}
	\label{riclemmaeq2}\end{equation}.
\label{riclemma} 
\end{riclemma}

\begin{proof}
For $K > 1$ we can write $\qgeps^+(\feps,\feps)$ as
	\begin{multline}
		\qgeps^+(\feps,\feps)
		= A + B
		:= \qgeps^+(\feps, \feps\1_{\feps\leq K})
		+ \qgeps^+ (\feps, \feps\1_{\feps > K})\\
		\leq K^{\frac{1}{2p'}} \qgeps^+(\feps, \feps^{\frac{p+1}{2p}})
		+ \frac{1}{\log K}\qgeps^+(\feps, \feps\log \feps).
	\label{qplussplit}\end{multline}

First, let $p\neq \infty.$ 
	\begin{multline}
		\|A\|_{L^p(\R^3)}
		 \leq K^{\frac{1}{2p'}}\frac{16}{\eps} \|\feps\|_{L^{\frac{2p}{p+1}}(\R^3)}\|\feps^{\frac{p+1}{2p}}\|_{L^{\frac{2p}{p+1}}(\R^3)}\\
		 \leq K^{\frac{1}{2p'}}\frac{16}{\eps}\|\feps\|_{L^{\frac{2p}{p+1}}(\R^3)}\|\feps\|_{L^1(\R^3)}^{\frac{p+1}{2p}}
		\leq K^{\frac{1}{2p'}}\frac{16}{\eps}\|\feps\|_{L^p(\R^3)}^\frac{1}{2},
	\end{multline}

where in the end we used the Riesz-Thorin Interpolation Theorem to get
 $\|\feps\|_{L^{\frac{2p}{p+1}}} \leq \|\feps\|_{L^1}^{\frac{1}{2}}\|\feps\|_{L^p}^{\frac{1}{2}} = \|\feps\|_{L^p}^{\frac{1}{2}}.$ 

For $B,$ we simply use Theorem \ref{thm1} with the coefficients $\left(p, 1, p\right):$
	\begin{multline}
		\|B\|_{L^p(\R^3)}
		\leq \frac{16}{\eps\log K} \|\feps\|_{L^p(\R^3)}\|\feps\log \feps\|_{L^1(\R^3)}\\
		\leq \frac{16}{\eps\log K} \|\feps\|_{L^p(\R^3)}\|f_0\|_{L\log L(\R^3)}.
	\end{multline}

Now, let $p=\infty.$ 
The proof is almost the same, but this time we bound $A$ by
	\begin{equation*}
		A = \qgeps^+(\feps, \feps\1_{\feps \leq K})
		\leq K\qgeps^+(\feps, 1),
	\end{equation*}
so that
	\begin{equation*}
		\|A\|_{L^\infty} 
		\leq K\frac{16}{\eps} \|\qgeps^+(\feps, 1)\|_{L^\infty}
		\leq K\frac{16}{\eps}\|\feps\|_{L^1}\|1\|_{L^\infty} 
		= \frac{16K}{\eps}.
	\end{equation*}

For $B,$ we follow the same steps as before:
	\begin{multline*}
		\|B\|_{L^\infty}
		\leq \frac{1}{\log K}\|\qgeps^+(\feps, \feps\log \feps)\|_{L^\infty}
		\leq \frac{16}{\eps\log K}\|\feps\|_{L^\infty} \|\feps\log \feps\|_{L^1} \\
		\leq  \frac{16}{\eps\log K}\|\feps\|_{L^\infty}\|f_0\|_{L\log L},
	\end{multline*}


as was to be shown.
\end{proof}

\section{Existence and uniqueness of $\bm{\feps}$}

The proof of existence and uniqueness is inspired by an existence proof from \cite{alonsogambatran}, in which the following theorem from \cite{bressan} is applied:

\newtheorem{bressan}{Theorem}[section]
\begin{bressan}
	Let $E$ be a Banach space, $F$ a bounded, convex and closed subset of $E$, and $Q: F \longrightarrow E$ an operator such that the following holds:
\begin{enumerate}[(i)]

\item Holder continuity: for all $f,h\in F,$
	\begin{equation}
		\|Q[f] - Q[h]\|_E \leq C\|f-h\|_E^\beta \hspace{.1cm} \text{ for some }\beta \in (0,1)
	\label{holder}\end{equation}
	
\item the subtangent condition: for all $f\in F,$
	\begin{equation}
		\liminf_{\delta\rightarrow 0^+} \frac{1}{\delta} dist_E(f + \delta Q[f], F) = 0
	\label{subtangent}\end{equation}
	
\item the one-sided Lipschitz condition: for all $f,h\in F,$
	\begin{equation}
		[Q[f] - Q[h], f-h] \leq C \|f-h\|_E \hspace{.1cm} \text{ for all } f, h \in F
	\label{onesidedlipschitz}\end{equation}
where $[\phi, \psi] : =  \lim_{\delta\rightarrow 0^+} \delta^{-1}(\| \psi + \delta\phi\|_E - \|\psi\|_E).$		
\end{enumerate}
Then, the equation 
	\begin{equation*}
		\begin{dcases}
			& \partial_t f = Q[f] \text{ on } [0, \infty)\times E\\
			& f(0) = \fo \geq 0 \in F \text{ on } \{0\}\times E
		\end{dcases}
	\label{bressaneq}\end{equation*}
has a unique solution, $f,$ which lies in $C^1((0,\infty), E)\cap C([0,\infty), F).$

\label{bressan}\end{bressan}

The proof of this theorem can be found in \cite{bressan} or \cite{alonsogambatran}. A direct application will allow us to prove the existence and uniqueness of $\feps$ by choosing an appropriate space $E$ and set $F.$

\newtheorem{existence}[bressan]{Theorem}
\begin{existence}
For $ 1\leq p \leq \infty$, let $ E: = L^1_2\cap L^p (\R^3)$ be a Banach space with the norm $\| f \|_E: = \|f\|_{L^1} + \|f\|_{L^p}$, and let $0\leq \fo \in  F\cap L^1_2\cap L\log L$ with $\|\fo\|_{L^1} = 1,$ where
	\begin{equation*}
		 F: = \left\{ f\in E: f\geq 0, \|f\|_{L^1} = \|\fo\|_{L^1} = 1, \|f\|_{L^p(\R^3)}\leq C\right\}
	\end{equation*}
for some $C >0.$
Then, there exists a unique solution, $\feps,$ to the Boltzmann equation (\ref{bteqgeps}) which lies in $C^1((0,\infty), E) \cap C^1([0,\infty), F)$ that preserves mass, momentum, energy, and whose entropy is bounded by the initial entropy.
\label{existence}\end{existence}

\begin{proof}
Let $\eps>0.$
Fix $p$, and for any $f \in F,$ define $Q[f]: = \qgeps(f,f).$ One can easily check that $F$ is bounded, convex and closed in $E,$ so it remains to show that (\ref{holder}), (\ref{subtangent}) and (\ref{onesidedlipschitz}) hold.

Using the bilinearity of $\qgeps$ and Theorem \ref{thm1}, for any $f,h \in F,$
	\begin{multline}
		\| \qgeps(f, f) - \qgeps(h,h)\|_E
		\leq \| \qgeps(f,f-h))\|_E 
		+ \|\qgeps(f-h,h)\|_E\\
		= \|\qgeps(f, f-h)\|_{L^1} + \|\qgeps(f, f-h)\|_{L^p}\\
		+ \|\qgeps(f-h, h)\|_{L^1} + \|\qgeps(f-h,h)\|_{L^p}\\
		\leq \frac{16}{\eps} \|f\|_{L^1}\|f -h\|_{L^1}
		+ \frac{16}{\eps} \|f\|_{L^p}\|f-h\|_{L^1}
		+ \frac{16}{\eps} \|f-h\|_{L^1}\|h\|_{L^1}\\
		+ \frac{16}{\eps} \|f-h\|_{L^p}\|h\|_{L^1}
		= \frac{16}{\eps} \|f-h\|_{L^1}\left(2 + \|f\|_{L^p}\right)
		+ \frac{16}{\eps} \|f-h\|_{L^p}\\
		\leq \frac{16}{\eps}(2+C)\|f-h\|_E.
	\label{continuity}\end{multline}
We have shown that $\qgeps$ is continuous; in particular, it is Holder continuous for any $\beta \in (0,1),$ therefore (\ref{holder}) holds.

Next, we prove subtangency. Fix $f\in F.$ For any $\beta > 0,$ it suffices to find $\delta_0 = \delta_0(f,\beta)>0$ and $\omega \in F$ such that for any $\delta \in (0,\delta_0),$ 
	\begin{equation}
		\frac{1}{\delta}\left\| f+ \delta \qgeps(f,f) -\omega \right\|_E < \beta.
	\label{subtan}\end{equation}

For $R, \delta>0$ we define $f_R$ and $\omega = \omega(f, \delta, R)$ in the same way as it is in Proposition 5.1 in \cite{alonsogambatran}: $f_R: = f\1_{B_R}$, $\omega: = f + \delta\qgeps(f_R, f_R).$ We will find suitable values for $R_0$ and $\delta_0$ and use them to define an $\omega_0 = \omega(f, \delta_0, R_0)$ for which (\ref{subtan}) will hold.

First, we show that $\omega \in F$. Indeed, for any $f\in F,$ $\delta, R >0,$ $\omega = (1 - 8\eps^{-1}\delta)f + \delta\qgeps^+(f_R, f_R) \geq (1 - 8\eps^{-1}\delta)f \geq 0$ whenever $\delta \leq \eps/8.$ Then $\|\omega\|_{L^1} = \int \omega = \int f + \delta\int\qgeps(f_R, f_R) = \int f = \int\fo = 1$ because $\int\qgeps(h,h) = 0$ for all $h \in L^1.$ Lastly, thanks to Theorem \ref{thm1},
	\begin{multline*}
		\|\omega\|_{L^p} 
		 \leq  \left(1-\frac{8}{\eps}\delta\right)\|f\|_{L^p} 
		+ \frac{16}{\eps}\delta\|f_R\|_{L^p}\|f_R\|_{L^1}\\
		\leq \left(1+ \frac{8}{\eps}\delta\right)\|f\|_{L^p}
		\leq 2C
	\end{multline*}
whenever $\delta < \eps/8.$ This shows that $\omega \in F$ for small enough $\delta.$
Now, 
	\begin{multline}
		\|f + \delta\qgeps(f,f) - \omega\|_E
		= \delta\|(\qgeps(f,f) - \qgeps(f_R,f_R)\|_E\\
		\leq \delta\frac{16}{\eps}\left(2+C\right)\|f-f_R\|_E
	\end{multline}
by (\ref{continuity}). Let $R = R_0$ be large enough so that $\|f-f_{R_0}\|_{E} \leq \delta\beta\frac{\eps}{16} \left(2+ C\right)^{-1}.$ Then if $0< \delta < \delta_0:= \min\{\eps/8,\beta\},$
	\begin{equation*}
		\frac{1}{\delta}\|f + \delta\qgeps(f,f) - \omega_0\|_E
		\leq \delta
		< \beta,
	\end{equation*}
so we have (\ref{subtan}).

Finally, we prove the one sided Lipschitz condition. For $f, h \in F,$ let $\phi:= \qgeps(f,f) - \qgeps(h,h)$ and $\psi := f-h.$ For $\delta >0$ and $\omega_f,\omega_h \in F,$
	\begin{multline*}
		\|\psi + \delta\phi\|_E 
		-\|\psi\|_E
		=\| f+\delta\qgeps(f,f) -h - \delta\qgeps(h,h)\|_E
		- \|f-h\|_E\\
		\leq \|f + \delta\qgeps(f,f) - \omega_f\|_E
		+ \|h + \delta\qgeps(h,h) - \omega_h\|_E
		+ \|\omega_f - \omega_h\|_E
		-\|f-h\|_E,
	\end{multline*}
Now, fix $\beta >0$. By the subtangency condition, there exists $\delta_0>0$ and $\omega_0^f,\omega_0^h \in F$ such that for any $0<\delta < \delta_0$,
	\begin{align*}
		& \|f + \delta\qgeps(f,f) - \omega_0^f\|_E \leq \frac{\delta\beta}{4},\\
		& \|h + \delta\qgeps(h,h) - \omega_0^h\|_E \leq \frac{\delta\beta}{4}.
	\end{align*}
By the construction of $\omega_0^f, \omega_0^h,$ for a large enough $R_0$ and $R>R_0,$
	\begin{multline*}
		\|\omega_0^f - \omega_0^h\|_E
		- \|f-h\|_E\\
		\leq \|f-h\|_E 
		+ \delta\|\qgeps(f_R, f_R) - \qgeps(h_R, h_R)\|_E
		-\|f-h\|_E\\
		\leq \delta\frac{16}{\eps}(2+C)\|f_R - h_R\|_E.
	\end{multline*}
If we choose $R_0$ large enough so that for $R>0,$
	\begin{equation*}
		\|f_R - h_R\|_E
		\leq \frac{\beta \eps}{2(2+C)},
	\end{equation*}
then the Lipschitz condition follows.

\end{proof}

\newtheorem{hthm'}[bressan]{Remark}
\begin{hthm'}
It is not hard to check that the H- theorem still holds for $\qgeps.$ This means that any nonnegative f that solves the Boltzmann equation (\ref{bteqgeps}) preserves, mass, momentum and kinetic energy, and has decreasing entropy.
Therefore, these properties did not need to be included in our choice of $F$.
\end{hthm'}


\section{A uniform bound on $\bm{\feps}$}

%
	
\newtheorem{integrability}[lemma3]{Theorem}
\begin{integrability}
	Let $\feps = \feps(v,t) \geq 0 \in L^1(\R^3)$ be a weak solution to the Boltzmann equation with nonnegative initial data $0\leq \feps(v,0) :=f_0 \in L^1_2(\R^3)\cap L^p(\R^3) \cap L\log L(\R^3),$ $1\leq p \leq \infty.$ 
Then $\feps$ remains in $L^p(\R^3)$, uniformly in $\eps$ and time. More specifically, 
	\begin{equation}
		\|\feps(\cdot, t)\|_{L^p(\R^3)} 
		\leq \max\left\{16e^{8\|\fo\|_{L\log L}}, \|\fo\|_{L^p}\right\}
		\text{ for all } t>0.
	\end{equation}

\label{integrability}\end{integrability}

\begin{proof}
Without loss of generality, $\|\feps\|_{L^1} = 1.$ Fix $\eps >0$ and let $K := e^{4\|\fo\|_{L\log L}} > 1$ (provided $f_0 \not\equiv 0$).
We begin with the case $p\neq\infty.$
By (\ref{riclemmaeq1}),
	\begin{multline}
		\eps\|\qgeps(\feps,\feps)(\cdot, t)\|_{L^p(\R^3)}\\
		\leq 16 K^{\frac{1}{2p'}}\|\feps(\cdot, t)\|_{L^p(\R^3)}^{\frac{1}{2}}
		+ 8\|\feps(\cdot, t)\|_{L^p(\R^3)}\left( \frac{2}{\log K}\|f_0\|_{L\log L(\R^3)}
		-  1\right)\\
		= 16 K^{\frac{1}{2p'}}\|\feps(\cdot, t)\|_{L^p(\R^3)}^{\frac{1}{2}}
		- 4\|\feps(\cdot, t)\|_{L^p(\R^3)}.
	\end{multline} 


Then
	\begin{multline}
		\eps p\|\feps(\cdot, t)\|_{L^p}^{p-1}\frac{d}{dt}\|\feps(\cdot, t)\|_{L^p(\R^3)}\\
		= \eps\frac{d}{dt}\|\feps(\cdot, t)\|_{L^p(\R^3)}^p
		= \eps\int \frac{d}{dt}\feps(v,t) \feps^{p-1}(v,t)\dv\\
		= \eps\int_{\R^3} \qgeps^+(\feps,\feps)(v,t) \feps^{p-1}(v,t)\dv
		- 8\int\feps^p(v,t)\dv\\
		\leq \eps\|\qgeps^+(\feps, \feps)(\cdot, t)\|_{L^p(\R^3)}\|\feps(\cdot, t)\|_{L^{p}(\R^3)}^{p-1}
		- 8\|\feps\|_{L^p}^p\\
		\leq  16K^{ \frac{1}{2p'}}\|\feps(\cdot,t)\|_{L^p(\R^3)}^{p-\frac{1}{2}}
		 -4 \|\feps(\cdot,t)\|_{L^p(\R^3)}^p.
	\label{normode}\end{multline}	 
Multiplying both sides of (\ref{normode}) by $\frac{1}{2}\|\feps(\cdot, t)\|_{L^p(\R^3)}^{\frac{1}{2} - p},$ 
	\begin{equation}
		\eps p\frac{d}{dt}\|\feps(\cdot, t)\|_{L^p}^{\frac{1}{2}}
		\leq 8 K
		-2 \|\feps(\cdot, t)\|_{L^p}^\frac{1}{2}.
	\label{odenorm2}\end{equation}
Then if $u(t): = \|\feps(\cdot, t)\|_{L^p}^{\frac{1}{2}}$, by (\ref{odenorm2}),
	\begin{equation}
		\begin{dcases}
			& u'(t) \leq \frac{8K}{\eps p}
			- \frac{2}{\eps p}u(t)\\
			& u(0) = \|\fo\|_{L^p}^{\frac{1}{2}}.
		\end{dcases}
	\label{ode}\end{equation}

A maximum principle then shows us that 
	\begin{multline*}
		\|\feps(\cdot, t)\|_{L^p(\R^3)} 
		= u^2(t)
		\leq \left(4K + (u(0) - 4K)e^{-\frac{2t}{\eps p}}\right)^2 \\
		\leq \max\left\{ 4K, \|f_0\|_{L^p(\R^3)}^{\frac{1}{2}}\right\}^2 
		\leq \max\left\{ 16e^{8\|\fo\|_{L\log L}}, \|f_0\|_{L^p(\R^3)}\right\}.
	\end{multline*}

Now, suppose that $p=\infty.$ By (\ref{riclemmaeq2}),
	\begin{equation*}
		\eps\|\qgeps^+(\feps,\feps)(\cdot, t)\|_{L^\infty(\R^3)}
		\leq  16K 
		+ 16\frac{\|f_0\|_{L\log L(\R^3)}}{\log K}\|\feps(\cdot, t)\|_{L^\infty(\R^3)}.
	\end{equation*}
Then
	\begin{multline*}
		\eps\frac{d}{dt}\feps(v,t)
		\leq \eps \|\qgeps^+(\feps,\feps)(\cdot, t)\|_{L^\infty(\R^3)}
		- 8\feps(v,t)\\
		\leq 16K
		+ 16\frac{\|f_0\|_{L\log L(\R^3)}}{\log K}\|\feps(\cdot, t)\|_{L^\infty}
		- 8\feps(v,t).
	\end{multline*}
In particular, by definition of supremum, this means that for all $v\in \R^3,$
		\begin{multline}
		\eps\frac{d}{dt}\feps(v,t)
		\leq 16K\\
		+ 8\left(2\frac{\|f_0\|_{L\log L(\R^3)}}{\log K}
		- 1\right)\feps(v,t)
		= 16K
		- 4\feps(v,t).
	\end{multline}
The maximum principle then tells us that
	\begin{equation*}
		\feps(v,t) 
		\leq (\fo - 4K)e^{-\frac{4t}{\eps}} + 4K,
	\end{equation*}
so that, finally,
	\begin{multline}
		\|\feps\|_{L^\infty}
		\leq (\|\fo\|_{L^\infty} - 4K)e^{-\frac{4t}{\eps}} + 4K
		\leq \max\left\{4e^{4\|\fo\|_{L\log L}}, \|\fo\|_{L^\infty}\right\}\\
		\leq \max\left\{16e^{8\|\fo\|_{L\log L}}, \|\fo\|_{L^\infty}\right\}.
	\end{multline}
\end{proof}

\bibliographystyle{plain} 
\bibliography{mybib}

\begin{thebibliography}{10}

\bibitem{advw}
R.~Alexandre, L.~Desvillettes, C.~Villani, and B.~Wennberg.
\newblock Entropy dissipation and long-range interactions.
\newblock {\em Arch. Ration. Mech. Anal. 152}, 4, 2000.

\bibitem{alonsocarneirogamba}
R.J. Alonso, E.~Carneiro, and I.M. Gamba.
\newblock Convolution inequalities for the boltzmann collision operator.
\newblock {\em Comm. Math. Phys. 298}, 2, 2010.

\bibitem{ricardo}
R.J. Alonso, I.M. Gamba, and M~Taskovic.
\newblock On $l^p$ propagation for solutions of the boltzmann equation for
  maxwell type interactions.
\newblock 2016.
\newblock in preparation.

\bibitem{alonsogambatran}
R.J. Alonso, I.M. Gamba, and M~Tran.
\newblock The cauchy problem for the quantum boltzmann equation for bosons at
  very low temperature.
\newblock 2016.
\newblock in preparation.

\bibitem{bgp}
A.~Bobylev, I.M. Gamba, and I.~Potapenko.
\newblock On some properties of the landau kinetic equation.
\newblock {\em J. Stat. Phys. 161}, 6, 2015.

\bibitem{bp}
A.~Bobylev and I.~Potapenko.
\newblock Monte carlo methods and their analysis for coulomb collisions in
  multicomponent plasmas.
\newblock {\em J. Comput. Phys. 246}, 2013.

\bibitem{bressan}
A.~Bressan.
\newblock Notes on the boltzmann equation.
\newblock 2005.
\newblock Lecture notes for a summer course, S.I.S.S.A. Trieste.

\bibitem{ccl}
E.~Carlen, M.C. Carvalho, and X.~Lu.
\newblock On strong convergence to equilibrium for the boltzmann equation with
  soft potentials.
\newblock {\em J. Stat. Phys.}, 135, 2009.

\bibitem{dl}
P.~Degond and B.~Lucquin-Desreux.
\newblock The fokker-planck asymptotics of the boltzmann collision operator in
  the coulomb case.
\newblock {\em Mth. Models Mth Appl. Sci.}, 2, 1992.

\bibitem{dv}
L.~Desvillettes.
\newblock Entropy dissipation estimates for the landau equation in the coulomb
  case and applications.
\newblock {\em J. Funct. Anal. 269}, 5, 2013.

\bibitem{dvvillani1}
L.~Desvillettes and C.~Villani.
\newblock On the spatially homogeneous landau equation for hard potentials. i.
  existence, uniqueness and smoothness.
\newblock {\em Comm. Partial Differential Equations 25}, 1-2, 2000.

\bibitem{dvvillani2}
L.~Desvillettes and C.~Villani.
\newblock On the spatially homogeneous landau equation for hard potentials. ii.
  h-theorem and applications.
\newblock {\em Comm. Partial Differential Equations 25}, 1-2, 2000.

\bibitem{dvvillani3}
L.~Desvillettes and C.~Villani.
\newblock On the trend to global equilibrium for spatially inhomogeneous
  kinetic systems: the boltzmann equation.
\newblock {\em Invent. Math. 159}, 2, 2005.

\bibitem{gambahaack}
I.M. Gamba and J.~Haack.
\newblock A conservative spectral method for the boltzmann equation with
  anisotropic scattering and the grazing collisions limit.
\newblock {\em J. Comput. Phys. 270}, 2014.

\bibitem{landau}
L.~Landau.
\newblock The transport equation in the case of coulomb interactions.
\newblock {\em JETP 7, 203,}, 1937.

\bibitem{villani}
C.~Villani.
\newblock On a new class of weak solutions to the spatially homogeneous
  boltzmann and landau equations.
\newblock {\em Arch. Rational Mech. Anal. 143}, 1998.

\end{thebibliography}

\end{document}